\documentclass[11pt]{article}

\usepackage{amsmath,amssymb,amsthm}
\usepackage{mathtools}
\usepackage{enumerate}
\usepackage{authblk}

\usepackage{mathptmx}
\usepackage[scaled=.90]{helvet}

\usepackage{setspace}

\usepackage{wrapfig}
\usepackage[margin=0.984in]{geometry}
\usepackage{leftidx}

\usepackage{float}
\usepackage[utf8]{inputenc}
\usepackage[hidelinks]{hyperref} 
\usepackage{caption}
\usepackage{import}

\usepackage{tikz, mathrsfs}

\usetikzlibrary{arrows,cd}

\tikzset{>=stealth'}

\def\arrowLengthDisplayStyle{4ex}
\def\arrowHeightDisplayStyle{.8ex}
\def\arrowSkipDisplayStyle{.5ex}
\def\arrowLengthTextStyle{3ex}
\def\arrowHeightTextStyle{.8ex}
\def\arrowSkipTextStyle{.4ex}
\def\arrowLengthScriptStyle{2.5ex}
\def\arrowHeightScriptStyle{.6ex}
\def\arrowSkipScriptStyle{.3ex}
\def\arrowLengthScriptScriptStyle{2ex}
\def\arrowHeightScriptScriptStyle{.4ex}
\def\arrowSkipScriptScriptStyle{.2ex}

\renewcommand{\mapsto}{\arrow{|->}}

\newcommand{\MakeTikzArrowWithSuperscriptSubscript}[4]
{
	\mathchoice
	{ 
		\hspace*{\arrowSkipDisplayStyle}
		\begin{tikzpicture}[baseline]
		\draw [#1] (0,\arrowHeightDisplayStyle) -- node [above] {$#2$} node [below] {$#3$} (#4 * \arrowLengthDisplayStyle, \arrowHeightDisplayStyle);
		\end{tikzpicture}
		\hspace*{\arrowSkipDisplayStyle}
	}
	{ 
		\hspace*{\arrowSkipTextStyle}
		\begin{tikzpicture}[baseline]
		\draw [#1] (0,\arrowHeightTextStyle) -- node [above] {$\scriptstyle #2$} node [below] {$\scriptstyle #3$} (#4 * \arrowLengthTextStyle, \arrowHeightTextStyle);
		\end{tikzpicture}
		\hspace*{\arrowSkipTextStyle}
	}
	{ 
		\hspace*{\arrowSkipScriptStyle}
		\begin{tikzpicture}[baseline]
		\draw [#1] (0,\arrowHeightScriptStyle) -- node [above] {$\scriptscriptstyle #2$} node [below] {$\scriptscriptstyle #3$} (#4 * \arrowLengthScriptStyle, \arrowHeightScriptStyle);
		\end{tikzpicture}
		\hspace*{\arrowSkipScriptStyle}
	}
	{ 
		\hspace*{\arrowSkipScriptScriptStyle}
		\begin{tikzpicture}[baseline]
		\draw [#1] (0,\arrowHeightScriptScriptStyle) -- node [above] {$\scriptscriptstyle #2$} node [below] {$\scriptscriptstyle #3$} (#4 * \arrowLengthScriptScriptStyle, \arrowHeightScriptScriptStyle);
		\end{tikzpicture}
		\hspace*{\arrowSkipScriptScriptStyle}
	}
}

\newcommand{\MakeTikzArrowWithCentralLabel}[3]
{
	\mathchoice
	{ 
		\hspace*{\arrowSkipDisplayStyle}
		\begin{tikzpicture}[baseline]
		\draw [#1] (0,\arrowHeightDisplayStyle) -- node [fill=white,inner sep=1pt] {$#2$} (#3 * \arrowLengthDisplayStyle, \arrowHeightDisplayStyle);
		\end{tikzpicture}
		\hspace*{\arrowSkipDisplayStyle}
	}
	{ 
		\hspace*{\arrowSkipTextStyle}
		\begin{tikzpicture}[baseline]
		\draw [#1] (0,\arrowHeightTextStyle) -- node [fill=white,inner sep=1pt] {$\scriptstyle #2$} (#3 * \arrowLengthTextStyle, \arrowHeightTextStyle);
		\end{tikzpicture}
		\hspace*{\arrowSkipTextStyle}
	}
	{ 
		\hspace*{\arrowSkipScriptStyle}
		\begin{tikzpicture}[baseline]
		\draw [#1] (0,\arrowHeightScriptStyle) -- node [fill=white,inner sep=1pt] {$\scriptscriptstyle #2$} (#3 * \arrowLengthScriptStyle, \arrowHeightScriptStyle);
		\end{tikzpicture}
		\hspace*{\arrowSkipScriptStyle}
	}
	{ 
		\hspace*{\arrowSkipScriptScriptStyle}
		\begin{tikzpicture}[baseline]
		\draw [#1] (0,\arrowHeightScriptScriptStyle) -- node [fill=white,inner sep=1pt] {$\scriptscriptstyle #2$} (#3 * \arrowLengthScriptScriptStyle, \arrowHeightScriptScriptStyle);
		\end{tikzpicture}
		\hspace*{\arrowSkipScriptScriptStyle}
	}
}

\def\arrow#1{\def\lastArrowStyle{#1}
	\futurelet\testchar\arrowMaybeStreched}
\def\arrowMaybeStreched{\ifx[\testchar \let\next\arrowStreched
	\else \let\next\arrowUnstreched \fi
	\next}

\def\arrowStreched[#1]{\def\lastArrowStrech{#1}
	\futurelet\testchar\arrowMaybeLabel}
\def\arrowUnstreched{\def\lastArrowStrech{1}
	\futurelet\testchar\arrowMaybeLabel}

\def\arrowMaybeLabel{\ifx^\testchar \let\next\arrowSuperscript
	\else \ifx_\testchar \let\next\arrowSubscript
	\else \ifx~\testchar \let\next\arrowCentralLabel
	\else \let\next\arrowNoLabel
	\fi
	\fi
	\fi
	\next}

\def\arrowSuperscript^#1{\def\lastArrowSuperscript{#1}
	\futurelet\testchar\arrowSuperMaybeSub}
\def\arrowSuperMaybeSub{\ifx_\testchar \let\next\arrowSuperscriptSubscript
	\else \let\next\arrowSuperscriptNoSubscript \fi
	\next}

\def\arrowSubscript_#1{\def\lastArrowSubscript{#1}
	\futurelet\testchar\arrowSubMaybeSuper}
\def\arrowSubMaybeSuper{\ifx^\testchar \let\next\arrowSubscriptSuperscript
	\else \let\next\arrowSubscriptNoSuperscript \fi
	\next}

\def\arrowSuperscriptSubscript_#1{\def\lastArrowSubscript{#1}
	\arrowDrawSupSub}
\def\arrowSuperscriptNoSubscript{\def\lastArrowSubscript{}
	\arrowDrawSupSub}
\def\arrowSubscriptSuperscript^#1{\def\lastArrowSuperscript{#1}
	\arrowDrawSupSub}
\def\arrowSubscriptNoSuperscript{\def\lastArrowSuperscript{}
	\arrowDrawSupSub}
\def\arrowNoLabel{\def\lastArrowSuperscript{}
	\def\lastArrowSubscript{}
	\arrowDrawSupSub}

\def\arrowCentralLabel~#1{\MakeTikzArrowWithCentralLabel{\lastArrowStyle}{#1}{\lastArrowStrech}}
\def\arrowDrawSupSub{\MakeTikzArrowWithSuperscriptSubscript{\lastArrowStyle}{\lastArrowSuperscript}{\lastArrowSubscript}{\lastArrowStrech}}

\tikzcdset{arrow style=tikz, diagrams={>=stealth'}, row sep=4em, column sep=4em}

\usepackage{pgfplots}

\usetikzlibrary{calc}
\usetikzlibrary{intersections, pgfplots.fillbetween}

\usepackage{lipsum}

\newcommand\blfootnote[1]{%
  \begingroup
  \renewcommand\thefootnote{}\footnote{#1}%
  \addtocounter{footnote}{-1}%
  \endgroup
}

\theoremstyle{definition}
\newtheorem{thm}{Theorem}[section]
\newtheorem{lemma}[thm]{Lemma}
\newtheorem{theorem}[thm]{Theorem}
\newtheorem{corollary}[thm]{Corollary}
\newtheorem{proposition}[thm]{Proposition}

\newtheorem{example}[thm]{Example}
\newtheorem{definition}[thm]{Definition}

\def\Map{\mathop{\rm Map}\nolimits}
\def\map{\mathop{\rm map}\nolimits}
\def\Fun{\mathop{\rm Fun}\nolimits}

\def\op{\mathop{\rm op}\nolimits}

\def\dirlim{\mbox{\,\hbox{lim}\kern-1.5em
                  \lower1.5ex\hbox{$\longrightarrow$}\,}}

\def\Z{{\mathbb{Z}}}
\def\R{{\mathbb{R}}}

\def\C{{\mathcal{C}}}
\def\D{{\mathcal{D}}}

\def\S{\mathcal{S}}

\def\Ab{\text{(Abelian groups)}}

\def\sSet{\text{(simplicial sets)}}

\def\Top{\text{(topological spaces)}}
\def\sTop{\text{(simplicial spaces)}}
\def\Cat{\text{(small categories)}}
\def\Gpoid{\text{(groupoids)}}
\def\topcat{\text{(topological categories)}}

\def\over{\text{(spaces over }\mathbb{R}\text{)}}

\def\ob{\mathrm{ob}}
\def\mor{\mathrm{mor}}
\def\nerve{\mathrm{N}}
\def\classifying{\mathrm{B}}

\def\reeb{\mathrm{R}}

\def\homology{\mathrm{H}}
\def\page{\mathrm{E}}

\def\sk{\mathrm{sk}}

\def\id{\mathrm{id}}

\def\pr{\mathrm{pr}}
\def\ev{\mathrm{eval}}

\def\Crm{\mathrm{C}}

\newcommand{\twopartdef}[4]
{
	\left\{
		\begin{array}{ll}
			#1 & \mbox{if } #2 \\
			#3 & \mbox{if } #4
		\end{array}
	\right.
}

\DeclareMathAlphabet{\mathcal}{OMS}{cmsy}{m}{n}	

\DeclareRobustCommand{\coprod}{\mathop{\text{\fakecoprod}}}
\newcommand{\fakecoprod}{%
  \sbox0{$\prod$}%
  \smash{\raisebox{\dimexpr.9625\depth-\dp0}{\scalebox{1}[-1]{$\prod$}}}%
  \vphantom{$\prod$}%
}	

\title{ Combinatorial models for topological Reeb spaces}
\author{Paul Trygsland}
\date{}

\begin{document}

\maketitle

\textbf{Abstract.} There are two rather distinct approaches to Morse theory nowadays: smooth and discrete. We propose to study a real valued function by assembling all associated sections in a topological category. From this point of view,~\emph{Reeb functions} on stratified spaces are introduced, including both smooth and combinatorial examples. As a consequence of the simplicial approach taken, the theory comes with a spectral sequence for computing (generalized) homology. We also model the homotopy type of Reeb graphs/ topological Reeb spaces as simplicial sets, which are combinatorial in nature, as opposed to the typical description in terms of quotient spaces. \blfootnote{paul.trygsland@ntnu.no}\blfootnote{Departement of Mathematical Sciences, Norwegian University of Science and Technology, 7491 Trondheim, Norway}

\section{Introduction}

Let~$X$ be a topological space and~$f\colon X \rightarrow \R$ a continuous function on it. A \emph{section}~$\sigma$ of~$f$ is a map~\hbox{$ [a,b] \rightarrow X$}, for some real numbers $a\leq b$, subject to~$f\circ \sigma (c)=c$. Two sections~$\sigma\colon [a,b]\rightarrow X$ and~$\rho\colon [b,c]\rightarrow X$, such that~$\sigma(b)=\rho(b)$, may be concatenated into a new section~$\rho\circ \sigma\colon [a,c]\rightarrow X$. This data defines the \emph{section category}~$\S_f$ associated to~$f$ which is in fact a topological category. The nerve construction thus provides a simplicial topological space~$\nerve\S_f$. We did not put any constraints on~$f$ as of yet. However, if the section category is to recover the homotopical information of~$X$ by realizing~$\nerve\S_f$, some assumptions are necessary. This should be considered motivation for the concept of \emph{Reeb functions} which requires~$f$ to be sufficiently `nice'. Examples include Morse functions on smooth manifolds and piecewise linear functions on CW complexes. I refer to Definition~\ref{definition:reebFunction} for a precise formulation.
\begin{theorem}
\label{intro:mainresult}
For any Reeb function~$f\colon X \rightarrow \R$, the realization of the nerve of the section category of~$f$ is weakly equivalent to~$X$, that is~$X\simeq |\nerve\S_f|$.
\end{theorem}

Ralph L. Cohen, John D. S. Jones and Graeme B. Segal prove a similar result for Morse functions in~\cite{cohen1995morse} as an attempt to better understand homotopical aspects of Morse theory~\cite{cohen1995floer}. A purely combinatorial analogue can be found in~\cite{nanda2018discrete} which covers the discrete Morse theory of Robin Forman~\cite{forman1998morse}. Our work can thus be described as an attempt to find a common framework including both smooth and combinatorial examples.

Any simplicial topological space comes with a spectral sequence for computing the generalized homology of its classifying space~\cite{segal1968classifying}. A shortcoming of the section category~$\S_f$ is that its classifying space is huge, hence nowhere near computationally feasible. Reeb functions provide a way to extract the essential information in~$\S_f$ into the much smaller \emph{critical subcategory}~$\C_f$ whose classifying space has unchanged homotopy type when compared to~$\S_f$. Computing the homology of~$X$ via~$\C_f$, as opposed to~$\S_f$, is analogous to how Morse and CW homology reduces the complexity of singular homology. I refer to Section~\ref{section:spectralsequence} for some basic algebraic properties together with a user-guide on how to carry out computations.

Consider a continuous function~$f\colon X\rightarrow \R$ on a topological space~$X$. The~\emph{topological Reeb space}~$\reeb_f$, often referred to as the Reeb graph, was introduced by Georges H. Reeb in~\cite{reeb1946points} to study singularities. Later on it was popularized in computer graphics due to the work of Y. Shinagawa, T. Kunii and Y. Kergosien~\cite{shinagawa1991surface}. Since then there has been several applications in shape analysis~\cite{biasotti2008reeb}. This advertises the need to better understand combinatorial properties of the topological Reeb space~$\reeb_f$, commonly constructed as a certain quotient space of~$X$ depending on the extra data that is~$f$.  I refer to categorified Reeb graphs~\cite{de2016categorified} and Mapper~\cite{singh2007topological} for related work. From the section category~$\S_f$ we define the \emph{combinatorial Reeb space} by first applying the nerve followed by taking path components level-wise~$\pi_0 \nerve\S_f$. It is important to note that this construction is no topological space, but rather a simplicial set. To compare topological and combinatorial Reeb spaces we make use of the fact that topological spaces and simplicial sets carry the same homotopical information: we identify the homotopy type of a simplicial set~$S$ with that of its geometric realization~$|S|$. The combinatorial and topological Reeb spaces of~$f$ do not have the same homotopy type in general. But if we restrict ourselves to Reeb functions, then they do agree.
\begin{theorem}
\label{intro:combinatorialReebIsClassicalReeb}
For any Reeb function~$f\colon X\rightarrow \R$, the simplicial set~$\pi_0\nerve\S_f$ has the same homotopy type as the topological space~$\reeb_f$; there is a zigzag of weak homotopy equivalences between~$|\pi_0\nerve\S_f|$ and~$\reeb_f$.
\end{theorem}

Topological Reeb spaces are not graphs in general (Example~\ref{example:ReebNotGraph}) and we might expect combinatorial Reeb spaces to have equally nasty homotopy types. But it turns out that combinatorial Reeb spaces are always weakly homotopic to graphs:
\begin{theorem}
\label{intro:combinatorialReebIsGraph}
The combinatorial Reeb space of any continuous function has the homotopy type of a~$1$--dimensional CW complex.
\end{theorem}

\textbf{Outline.} Section~\ref{sec:reebfunctions} is all about Reeb functions~$f\colon X\rightarrow \R$. To better illustrate the theory we first restrict ourselves to functions on~$\Crm^1$--manifolds in Section~\ref{subsec:reebC1} before handling more general stratified spaces in Section~\ref{section:stratifiedSpaces}. Results that do not hinge upon any simplicial structure are proven along the way. In Section~\ref{section:sectioncategory} we formally define the topological section category associated to a continuous function as well as the critical subcategory and other intermediate subcategories. Some simplicial background is then provided in Section~\ref{section:moreSimplicialStuff} before proving Theorem~\ref{thm:mainresult}, which implies Theorem~\ref{intro:mainresult}. The spectral sequence associated to section categories, as well as critical subcategories, is discussed in Section~\ref{section:spectralsequence}. General algebraic properties are deduced in Section~\ref{subsec:sectionsequence}, whereas Section~\ref{subsec:criticalsequence} is concerned with how to use the critical subcategory for numerical computations. In the remaining Section~\ref{section:reeb} we introduce combinatorial Reeb spaces. More background on simplicial sets is presented in Section~\ref{section:thmA} before proving Theorems~\ref{intro:combinatorialReebIsClassicalReeb} and~\ref{intro:combinatorialReebIsGraph} in Section~\ref{subsec:combistop} and Section~\ref{subsec:combisgraph}, respectively.

\textbf{Notation.} Categories of familiar objects are put inside parentheses, e.g.~$\Top$. The set of morphisms between objects~$x,y$ in a category is denoted~$\Map(x,y)$. In the case of topological spaces~$\map(X,Y)$ reads the topological space of continuous functions from~$X$ to~$Y$. The standard~$n$--simplex~$\Delta^n$ is modeled as the convex hull of the standard basis vectors in~$\R^{n+1}$. The~$1$--simplex will also be represented as the unit interval~$I$. We denote by~$[n]$ the category generated by the directed graph
\[
0\rightarrow 1 \rightarrow \dots \rightarrow n
\]
on~$n$ arrows. In particular,~$[0]$ is the trivial one object category and~$[1]$ is the category on two objects~$0$ and~$1$ connected by a non-trivial arrow~$0\rightarrow 1$.

\section{Reeb functions}
\label{sec:reebfunctions}

We shall clarify what it means for a function~$f\colon X \rightarrow \R$ to be a Reeb function. In this work, a stratified space is built out of~$\Crm^1$--manifolds, which will be covered more in depth later on. Hence we start out by restricting ourselves to the simplest spaces, namely the~$\Crm^1$--manifolds, in Section~\ref{subsec:reebC1}. Thereafter we move on to the more general stratified spaces in Section~\ref{section:stratifiedSpaces}. The final Proposition~\ref{proposition:modifiedFlowLines} is utilized many times throughout the paper.

\subsection{Reeb functions on~$\Crm^1$--manifolds}
\label{subsec:reebC1}

A continuous function~$f$ is said to be \emph{proper} if the preimage of compact is compact.

\begin{definition}
\label{definition:reebFunctionC1}
Let~$M$ be a~$\Crm^1$--manifold and~$f\colon M\rightarrow \R$ a~$\Crm^1$--function. Then~$f$ is a~\emph{Reeb function} if
\begin{enumerate}[i)]
\item the subspace of critical values of~$f$ is discrete inside~$\R$ and
\item the restriction of~$f$ to each component of~$M$ is proper.
\end{enumerate}
\end{definition}

Recall that a~$\Crm^1$--function~$f\colon M\rightarrow \R$ has a differential~$df\colon M\rightarrow \mathrm{T}^\ast M$ which is a section of the cotangent bundle;~1-form. Let us think of~$df$ in terms of its~\emph{gradient vector field}: Pick an inner product~$\langle -,- \rangle$ on~$\mathrm{T}M$, and characterize~$\mathrm{grad}(f)\colon M\rightarrow \mathrm{T}M$ by~$\langle \mathrm{grad} (f), \mathbf{v} \rangle = df(\mathbf{v})$ for all vector fields~$\mathbf{v}\colon M\rightarrow \mathrm{T}M$. The integral curves of a vector field~$\mathbf{v}\colon M\rightarrow \mathrm{T}M$ are the~$\Crm^1$--curves~$l\colon (\alpha, \omega)\rightarrow M$, allowing~$\pm \infty$, satisfying~$\frac{dl}{dt}=\mathbf{v}_{l(t)}$. A~\emph{local flow} on~$M$ is a map~$\Psi\colon U\rightarrow M$, defined on an open neighborhood~$U$ of~$ \{0\}\times M$ in~$\R\times M$, such that~$U\cap (\R\times \{p\})$ is an interval for which~$\Psi$ restricts to an integral curve. The maximal integral curves~$l_p$ of~$\mathbf{v}\colon M\rightarrow \mathrm{T}M$ form the maximal flow~$\Psi_{\mathbf{v}}(p,t)=l_p(t)$. It is maximal in the sense that there are no other local flows which contains the domain of~$\Psi_{\mathbf{v}}$. For this maximal flow, let us write~$(\alpha_p,\omega_p)=U\cap (\R\times\{p\})$, allowing~$\pm \infty$ as endpoints. Then~$l_p\colon (\alpha_p,\omega_p)\rightarrow M$ is the maximal integral curve subject to~$l_p(0)=p$. If an integral curve passes through a point~$q$ with~$\mathbf{v}(q)=0$ then~$l\colon \R\rightarrow M$,~$t\mapsto q$ is the obvious solution. This means, conversely, that all other integral curves are immersions. They do not have to be embeddings, in general. But it is the case whenever~$\mathbf{v}=\mathrm{grad}(f)$ for a function~$f$ as above:
\begin{align*}
\frac{d (f\circ l)}{dt}&= df_{l(t)}(\frac{dl}{dt})
=\langle \mathrm{grad}(f),\mathrm{grad}(f) \rangle_{l(t)}
\end{align*}
which is greater than zero so that~$f\circ l$ and hence~$l$ are both injective. The existence of integral curves follows by solving local differential equations. In fact, vector fields and maximal flows are in one-to-one correspondence~\cite[p.~82-83]{brocker1982introduction}. I will refer to the maximal integral curves of~$\mathrm{grad}(f)$ as the \emph{flow-lines} of~$f$.\\

\begin{definition}
Let~$f\colon X\rightarrow \R$ be a continuous function. A \emph{section} of~$f$ is a continuous function~$\sigma\colon [a,b]\rightarrow \R$ such that~$f\circ \sigma$ is the inclusion~$[a,b]\hookrightarrow \R$.
\end{definition}
The next assertion tells us how to continuously pick sections of Reeb functions, a property which will turn out to be extremely useful.

\parbox{\linewidth}{ 
\begin{proposition}
\label{proposition:modifiedFlowLinesManifold}
Let~$f\colon M\rightarrow \R$ be a Reeb function. For any pair~$c<d$ of successive critical values, there is a continuous function~$ g\colon  [c,d] \times f^{-1}(c,d)\rightarrow X$ such that for all~$x$ the curve~$g_x=g(-,x)$
\begin{enumerate}[i)]
\item is a section;~$f\circ g_x(t) = t$, and
\item pass through~$x$ at~$f(x)$;~$g_x(f(x))=x$.
\end{enumerate}
\end{proposition}
}
\begin{proof}
The idea is simple: We would like to reparametrize the flow-lines of~$f$. Let~$\Psi\colon U \rightarrow f^{-1}(c,d)$ be the maximal flow of~$f$ restricted to~$f^{-1}(c,d)$. For every flow-line~$l_x\colon (\alpha_x,\omega_x)\rightarrow f^{-1}(c,d)$, the preceding discussion implies that the composition~$f \circ l_x\colon (\alpha_x,\omega_x) \rightarrow \R$ is injective. And so it defines a~$\Crm^1$--isomorphism~$h_x\colon (\alpha_x,\omega_x) \rightarrow (f\circ l_x (\alpha_x), f\circ l_x (\omega_x))$. The target must necessarily equal~$(c,d)$, independently of~$x$: there are no critical points in~$f^{-1}(c,d)$ and so an integral curve must meet every fiber. If not, one could have extended it by solving a local differential equation, contradicting the maximality of~$\Psi$. The map~$h\colon U\rightarrow (c,d)\times f^{-1}(c,d)$,~$(a,x)\mapsto (h_x (a),x)$ is a~$\Crm^1$--diffeomorphism. Its inverse is explicitly given by~$(a,x)\mapsto (h_x^{-1}(a),x)$. Define 
\[
\tilde{g}\colon (c,d)\times f^{-1}(c,d) \xrightarrow{h^{-1}} U\xrightarrow{\Psi} f^{-1}(c,d),
\]
then~$l_x=\Psi(-,x)$ implies that the restriction~$\tilde{g}_x=\tilde{g}(-,x)$ is equal to~$\tilde{g}_x(t)=l_x(h_x^{-1}t)$ and thus
\[
f\circ \tilde{g}_x(t)=(f\circ l_x) (h_x^{-1}(t))=t.
\]
Also, the equation~$x=l_x(0)$ implies
\[
\tilde{g}(f(x),x)=l_x(h_x^{-1}\circ f\circ l_x(0))=l_x(0)=x.
\]
Hence the map~$\tilde{g}$ satisfies the asserted properties i) and ii).

The proof will be complete once we have extended the map~$\tilde{g}$ to~$[c,d]\times f^{-1}(c,d)$. One can alternatively view~$\tilde{g}$ as a map~$f^{-1}(c,d)\rightarrow \map ((c,d),f^{-1}[c,d])$, utilizing the right adjoint. In fact, the two properties of~$\tilde{g}$ above tells us that its adjoint factorizes through the subspace~$\mathrm{Flow}_f(c,d)$, of~$\map ((c,d),f^{-1}[c,d])$,  consisting of flow-lines reparametrized as sections~$(c,d)\rightarrow f^{-1}[c,d]$. So the map~$\tilde{g}$ might as well be interpreted as a map~$f^{-1}(c,d)\rightarrow \mathrm{Flow}_f(c,d)$. Since~$f$ is Reeb, hence proper on connected components, the preimage~$f^{-1}[c,d]$ is a disjoint union of compact topological spaces. Consequently any flow-line of the form~$\tilde{g}_x(c,d)\rightarrow f^{-1}[c,d]$ can be extended uniquely to a section~$g_x\colon [c,d]\rightarrow f^{-1}[c,d]$. In other words, there is a function~$e\colon\mathrm{Flow}_f(c,d)\rightarrow \S_f(c,d)$ that extends reparametrized low-lines on~$(c,d)$ to sections on~$[c,d]$. The rather tedious task of demonstrating the continuity of~$e$ is all that remains. For then the composition
\[
f^{-1}(c,d)\xrightarrow{\tilde{g}} \mathrm{Flow}_f(c,d)\xrightarrow{e} \S_f(c,d)
\]
admits an adjoint~$g\colon [c,d]\times f^{-1}(c,d)\rightarrow X$ satisfying the asserted properties.

For every~$a\leq b$ in~$[c,d]$ and~$V$ open in~$M$, denote by~$\Crm([a,b],V)$ the subbasis element whose points are the maps~$\rho\colon[a,b]\rightarrow M$ for which~$\rho([a,b])\subset V$. Then the collection of all~$\Crm([a,b],V)\cap \S_f(c,d)$ is a subbasis for~$\S_f(c,d)$. Similarly, the collection of all~$\Crm([a,b],V)\cap \mathrm{Flow}_f(c,d)$, with~$c<a\leq b<d$, is a subbasis for~$\mathrm{Flow}_f(c,d)$. We need only verify that every preimage of the form~$e^{-1}(\Crm([a,b],V)\cap \S_f(c,d))$ is open. This is trivial whenever~$c<a$ and~$b<d$, for then the preimage~$e^{-1}(\Crm([a,b],V)\cap \S_f(c,d))$ is the set~$\Crm([a,b],V)\cap \mathrm{Flow}_f(c,d)$ which is open. To complete the proof, we will assume~$a=c$ and~$b<d$ henceforth: The case~$a>c$ and~$b=d$ is similar, whereas~$a=c$ and~$b=d$ is a special case of the former.

Take an arbitrary flow-line~$\tilde{g}$ in~$e^{-1}(\Crm([c,b],V)\cap \S_f(c,d))$. Let~$g=e\tilde{g}$ be the extension to~$[c,d]$ so that~$g(c)$ is the limit point of~$\tilde{g}$ in~$f^{-1}(c)$. We need only prove that there is an open neighborhood~$N$, of~$\tilde{g}$, inside~$e^{-1}(\Crm([c,b],V)\cap \S_f(c,d))$. To construct such a neighborhood we first pick a monotone sequence~$(a_n)$ in~$(c,b]$ converging to~$c$. Ehresmann's fibration theorem~\cite{ehresmann1950connexions} provides a~$\Crm^1$--diffeomorphism~$E_{a'}$ over~$\R$:
\begin{center}
		\begin{tikzcd}
			f^{-1}(c,d)  \arrow[r, "E_{a'}"] \arrow[dr, "f"] & f^{-1}(a')\times (c,d) \arrow[d, "\mathrm{pr}_2"] \\
			
			                                            & \R                      
		\end{tikzcd}
\end{center}
for every real number~$b<a'<c$. The elementary opens in~$ (f^{-1}(a')\cap V)\times (c,d)$ are all of the form~$B\times (c',d')$ where $B$ is an open ball. Since every restriction~$g|_{[a_n,b]}$ has compact image and~$g$ maps into~$V$, there are cylinders~$C_n=E_{b}^{-1}(B_n\times [a_n,b])$ contained in~$V$ with the property that~$N_n=\Crm([a_n,b],C_n)$ is a neighborhood of~$\tilde{g}$. Moreover, it is safe to assume that the radius of~$B_n$ tends to zero as~$n$ goes to infinity: If $B'$ is a ball contained inside~$B$, and~$B\times (c',d')$ maps into~$V$ under Ehresmann's~$\Crm^1$--diffeomorphism, then surely so does~$B'\times (c',d')$. I claim that we can choose~$N=N_{n_0}$ for some~$n_0$. Assume conversely that this is not the case. Then no~$N_n$ is contained in~$e^{-1}(\Crm([c,a'],V)\cap \S_f(c,d))$. So for every~$n$ there is a flow-line~$\rho_n$ and a real number~$a_n'$ in~$[c,a_n]$ such that~$\rho_n(a_n')$ is in the complement of~$V$. But the sequence~$(\rho_n(a_n'))$ converges to the point~$g(c)$--inside~$V$--by construction, a contradiction.
\end{proof}

\subsection{Extension to stratified spaces}
\label{section:stratifiedSpaces}

There are several notions of `stratified spaces' around. One of which is the locally cone-like spaces dating back to R. Thom's work in the late 60s~\cite{thom1969ensembles}. A more recent reference is~\cite{goresky1983intersection}. For any topological space~$Z$, we will denote the open cone~$Z\times [0,1)/ Z\times 0$ by~$\Crm(Z)$. As an example the open cone on the~$(n-1)$--sphere is the open~$n$--disk. A \emph{filtration-preserving map} between two filtrations~$X_0\subset X_1\subset\cdots$ and~$Y_0\subset Y_1\subset\cdots$ of topological spaces, consists of continuous functions~$g_n\colon X_i\rightarrow Y_i$ which commutes with the inclusions:~$g_{n+1}\circ (X_n\subset X_{n+1})=(Y_n\subset Y_{n+1})\circ g_n $. 
\begin{definition}
\label{definition:stratifiedSpace}
An~$n$--dimensional \emph{stratification} on a topological space~$X$ is a filtration
\[
\emptyset =X_{-1}\subset X_0\subset X_1\subset\dots \subset X_n=X
\]
satisfying: i) every~$i$th \emph{stratum}~$S_i=X_i\setminus X_{i-1}$ is an~$i$--dimensional~$\Crm^1$--manifold and ii) for every point~$x$ in~$S_i$ there exists an open neighborhood~$U$ about~$x$, an~$(n-i-1)$--dimensional stratified space~$Z$ and a filtration-preserving homeomorphism~$h:U\simeq\R^i\times \Crm(Z)$. The restriction which takes~$U\cap S_{i+j+1}$ to~$\R^i\times \Crm (Z_j-Z_{j-1})$, and its inverse, are both required to be~$\Crm^1$. We say that a topological space together with an~$n$--dimensional stratification is a \emph{stratified space} of dimension~$n$.
\end{definition}

Finite-dimensional stratified spaces form a category by only considering filtration-preserving maps. Include filtered colimits to get a more general notion of~\emph{stratified spaces}, allowing infinite filtrations. Every~ CW complex~$X$ fits into this larger category: The~$i$th stratum of~$X$ is the disjoint union of its (open)~$i$--cells. In particular, every weak homotopy type can be represented by such a space. \\

A continuous function~$f:X\rightarrow \R$, from a stratified space~$X$, is~\emph{strata-wise}~$\Crm^1$ if it is~$\Crm^1$ when restricted to each stratum. A point~$x$ in the~$i$th stratum of~$X$ is \emph{critical} if it is a critical point of the~$\Crm^1$--map~$f|_{S^ i}$.\\

\begin{example}
\label{example:existenceCritical}
For a given stratifiable space~$X$, the definition of a strata-wise~$\Crm^1$--function depends on the choice of stratification. Because of this we can always assume a Reeb function to have critical values. Indeed, let~$f\colon X\rightarrow \R$ be a Reeb function for which there are no critical values. Then we slightly modify the stratified structure on~$X$: refine the already existing structure by dividing every stratum~$S$ into the three strata~$f|_S^{-1}(-\infty,0)$,~$f|_S^{-1}(0)$ and~$f|_S^{-1}(0,\infty)$. Then~$f$ is still a Reeb function on~$X$ with this choice of stratification. Moreover, we now have a critical value~$0$.
\end{example}

We extend Definition~\ref{definition:reebFunctionC1} from differentiable manifolds to stratified spaces in the following way:

\begin{definition}
\label{definition:reebFunction}
Let~$X$ be a stratified space and~$f\colon X\rightarrow \R$ a strata-wise~$\Crm^1$--function. We say that~$f$ is a \emph{Reeb function} if
\begin{enumerate}[i)]
\item the subspace of critical values of~$f$ is discrete inside~$\R$ and
\item for any connected component~$C$ of some stratum, the restriction of~$f$ to the closure of~$C$, in~$X$, is proper.
\end{enumerate}
\end{definition}

For the purpose of proving Thoerem~\ref{intro:mainresult}, this will turn out to be a satisfactory extension. In particular, there is the stratified version of Proposition~\ref{proposition:modifiedFlowLinesManifold}.

\begin{proposition}
\label{proposition:modifiedFlowLines}
Let~$f\colon X\rightarrow \R$ be a Reeb function. For any pair~$c<d$ of successive critical values, there is a continuous function~$ g\colon  [c,d] \times f^{-1}(c,d)\rightarrow X$ which satisfies
\begin{enumerate}[i)]
 \item every~$g_x\colon [c,d]\rightarrow X$,~$g_x(t)=g(t,x)$ is a section and \item ~$g(f(x),x)=x$.
\end{enumerate}
\end{proposition}
\begin{proof}
For a general stratified space~$X$, and Reeb function~$f\colon X\rightarrow \R$, let~$i_1,i_2,\dots$ denote the indices of the non-empty strata. The proof is by induction on~$i_n$. To ease notation I will simply reindex~$i_n\mapsto n$. Define~$f_n$ to be the restriction of~$f$ to~$X_n$. For~$n=0$ there is nothing to prove if~$X_0$ is~$0$--dimensional, otherwise the base case follows by Proposition~\ref{proposition:modifiedFlowLinesManifold}. Assume that a function~$g_{n-1}\colon f_{n-1}^{-1}(c,d)\times [c,d]\rightarrow X_{n-1}$ is constructed to satisfy the assertion. We shall modify the gradient vector field on the~$n$th stratum to take into account the flow on lower dimensional strata. Definition~\ref{definition:stratifiedSpace} tells us that a point~$x$ in~$S_i \cap X_n \cap f^{-1}(c,d)$ admits a neighborhood~$N_x$, contained in~$f^{-1}(c,d)$, of the form~$\R^i\times \Crm (Z)$ with~$Z$ an~$(n-i-1)$--dimensional stratified space. We shall define a vector field on each~$N_x\cap S_n$ to obtain a new vector field on all of~$S_n$ via a partition of unity.

If~$i<n$, then the~$(n-i-1)$st stratum of~$Z$, which is locally~$\Crm^1$--diffeomorphic to~$\R^{n-i-1}$, indicates the intersection between~$\R^i\times \Crm (Z)$ and~$S_n$. So the intersection of~$N_x$ and the~$n$th stratum~$S_n$ may be covered by opens~$N_{x,j_x}\simeq \R^i\times \Crm(\R^{n-i-1})$. Let us construct a vector field on one such neighborhood~$N$ which meets~$S_n$ in $U\simeq \R^i \times \R \times \R^{n-i-1}$ and~$S^i$ in~$V=N\cap S_i\simeq \R^i$. There is a~$\Crm^1$--map~$U\rightarrow V$ which is given by the projection~$\pr_1\colon \R^i\times \R\times \R^{n-i-1}\rightarrow \R^i$ in coordinates. The induced map~$\mathrm{T}\pr_1$ on tangent spaces admits a right inverse~$v\mapsto (v,0,0)$. Hence a vector field on~$V$ defines a vector field on~$U$. In particular, the vector field corresponding to an appropriate restriction of~$g_{n-1}$ defines a vector field~$\mathbf{u}\colon U\rightarrow \mathrm{T}U$. Notice that an integral curve~$l$ of~$\mathbf{u}$ cannot have a limit point in~$X_{n-1}\cap f^{-1}(c,d)$ since~$g_{n-1}$ is a family of~$\Crm^1$ sections. For every~$x$ in~$X_{n-1}$, also contained in the closure of~$S_n$, we associate such a vector field~$\mathbf{u}_x\colon U_x\rightarrow \mathrm{T}U_x$. Otherwise, if~$i=n$ and~$x$ is not contained in any such~$U_x$, then~$N_x\simeq \mathbb{R}^n$ and we simply restrict the gradient vector field on~$S_n$ to~$N_x$. 

To define a vector field on all of~$S_n$, we cover~$S^n$ with a family of opens~$(U_\alpha)$ as described above and pick a partition of unity~$(\rho_{U_\alpha})$. The formula~$\mathbf{v}=\sum_{\alpha} \rho_{U_{\alpha}} \mathbf{u}_{\alpha}$ defines a vector field~$S_n\rightarrow \mathrm{T}S_n$. Notice how~$df(\mathbf{v})$ is non-zero everywhere precisely because each~$df(\mathbf{u}_\alpha)$ is non-zero everywhere. The corresponding maximal local flow thus results in a map~$g_n\colon [c,d]\times f|_ {S_n}^{-1}(c,d)\rightarrow X$. Combine~$g_{n-1}$ and~$g_{n}$ to define the parametrized family~$g\colon [c,d]\times f_n^{-1}(c,d)\rightarrow X_n$
\[
g(t,x)=\twopartdef{g_{n-1}(t,x)}{x\in X_{n-1}}{g_n(t,x)}{x\in S_n}
\]
of sections.
\end{proof}

We end this entire section by proving a lemma. The result is analogous to two basic Morse lemmas that utilizes flow-lines to produce deformation retracts.

\begin{lemma}
\label{lemma:oneCriticalStratified}
Let~$f\colon X\rightarrow \R$ be a Reeb function with at most one critical value. Then the inclusion~$f^{-1}a\hookrightarrow X$ is a homotopy equivalence for all~$a$ if there is no critical value, otherwise it is a homotopy equivalence for~$a$ equal to the critical value.
\end{lemma}
\begin{proof}
Define a filtration~$X_n=f^{-1}[a-n,a+n]$,~$n\geq 0$, on~$X$. Given that~$X$ is the homotopy colimit over~$X_n$, it suffices to prove that the inclusion~$i_n \colon f^{-1}a\hookrightarrow f^{-1}[a-n,a+n]$ is a weak homotopy equivalence. The inclusion certainly factorizes
\[
f^{-1}a \xhookrightarrow{j_n} f^{-1}[a-n,a]\xhookrightarrow{k_n} f^{-1}[a-n,a+n]
\]
and we will only argue that~$j_n$ is a weak homotopy equivalence. For the case of~$k_n$ is similar.

To every point~$x$, in~$f^{-1}[a-n,a]$, we associate the reparametrized flow-line~$g_x\colon [a-n,a]\rightarrow X$ through~$x$ provided by Proposition~\ref{proposition:modifiedFlowLines}. If~$x$ is in~$f^{-1}a$, then~$g_x\colon \{a \}\rightarrow X$ is the trivial section at~$x$. Define a retract~$r_n$ of~$j_n$ by the formula~$r_n(x)=g_x(a)$. This defines a homotopy equivalence due to the homotopy
\[
H(x,t)= g_x(ta+(1-t)f(x))
\]
from~$H(x,0)=x$ to~$H(x,1)=j_n\circ r_n(x)$.
\end{proof}

\section{The section category and its classifying space}
\label{section:sectioncategory}

In Section~$\ref{subsec:sf}$ we define the section category~$\S_f$ of a continuous function~$f$. Also, if~$f \colon X\rightarrow \R$ is a Reeb function, then every subset~$A$ of~$\R$ which contains the critical values of~$f$ defines a subcategory~$\C_f^A$ of~$\S_f$. Section~$\ref{section:moreSimplicialStuff}$ is included for the reader that would like some background on simplicial sets. Thereafter Theorem~\ref{intro:mainresult} is deduced from the stronger Theorem~\ref{thm:mainresult} in Section~\ref{sec:mainresult}.

\subsection{The section category}
\label{subsec:sf}

Let us first agree on the meaning of a `topological category'. There are two different flavors: categories enriched in topological spaces and categories internal to topological spaces. In this paper a topological category is to be understood in the latter sense, following G. Segal~\cite{segal1968classifying}. A category~$\C$ can be described in terms of four structural maps: If~$\ob\C$ is the set of objects;~$\mor\C$ the set of morphisms; then they are source and target~$s,t\colon \mor\C \rightarrow \ob\C$, injection of objects as identity morphisms~$i\colon \ob\C\rightarrow \mor\C$ and composition~$\circ\colon  \mor\C\times_{\ob\C }\mor \C\rightarrow \mor\C$. The set~$ \mor\C\times_{\ob\C }\mor \C$ is the pullback obtained from the source and target; consists of pairs~$(m,m')$ of morphisms for which~$s(m')=t(m)$ such that~$m'\circ m$ is defined. A category~$\C$ is a \emph{topological category} if both~$\ob\C$ and~$\mor\C$ are equipped with topologies and the four structural maps~$s$,~$t$,~$i$ and~$\circ$ are all continuous. Any topological space~$X$ defines a trivial topological category~$\underline{X}$ whose object space and morphism space are both equal to~$X$. The structural maps~$s,t,i$ all agree with the identity on~$X$, whereas composition is the homeomorphism from the diagonal on~$X$ to~$X$.\\

Assume that a continuous function~$f\colon X\rightarrow \R$ from a topological space~$X$ is given. Recall that a \emph{section} of~$f$ is a continuous function~$\sigma\colon [a,b]\rightarrow X$ such that~$f\circ \sigma\colon [a,b]\rightarrow \R$ is the inclusion. Arrange all of the sections in the \emph{space of all sections}~$\mor \S_f= \coprod_{a\leq b} \S_f [a,b]$, ranging over all pairs~$a\leq b$ in~$\R$, equipped with the disjoint union topology. Notice how~$f^{-1}a$ and~$\S_f[a,a]$ are canonically homeomorphic. Hence the object space~$\ob\S_f=\coprod_{a\in\R} f^{-1}a$ comes with an evident map~$i\colon\ob\S_f\rightarrow \mor\S_f$. Source and target maps~$s,t\colon \mor\S_f\rightarrow \ob\S_f$ are obtained by restricting the evaluation~$\ev\colon \S_f[a,b]\times [a,b]\rightarrow X$ to~$a$ and~$b$, respectively. If~$\sigma\colon[a,b]\rightarrow X$ is a section, then~$s(\sigma)=\sigma(a)$ and~$t(\sigma)=\sigma(b)$. Concatenation defines canonical maps~$\S_f[b,c]\times_{f^{-1}b}\S_f[a,b]\rightarrow \S_f[a,c]$:
\[
\rho\circ \sigma(r)=\twopartdef{\sigma(r)}{a\leq r \leq b}{\rho(r)}{b\leq r \leq c}
\]
From which a composition~$\circ \colon \mor \S_f\times_{\ob\S_f}\mor\S_f\rightarrow \mor\S_f$ is deduced.

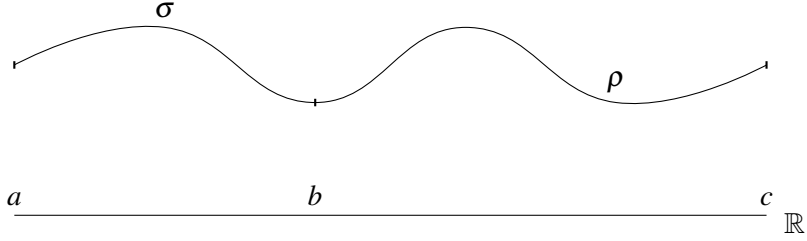
\begin{figure}[h]
\centering			
\begin{tikzpicture}
\draw [-] plot [smooth, tension=0.8] coordinates {(0,0) (2,0.5)  (4,-0.5)  (6,0.5) (8,-0.5)  (10,0)};

\draw [-] (0,-2) -- (10,-2);

\draw [-,thick] (0,-0.05) -- (0,0.05);

\draw [-,thick] (4,-0.55) -- (4,-0.45);

\draw [-,thick] (10,-0.05) -- (10,0.05);

\node [right] at (10.1,-2.1) {$\R$ };

\node [above] at (0,-2) {$a$ };

\node [above] at (4,-2) {$b$ };

\node [above] at (10,-2) {$c$ };

\node [above] at (2,0.5) {$\sigma$ };

\node [above] at (8,-0.5) {$\rho$ };

\end{tikzpicture}	
\caption{The composition of two sections $\sigma$ and $\rho$ satisfying $s(\rho)=t(\sigma)$.}	
\label{fig:comp}
\end{figure}

It is straightforward to check that~$\circ$ is associative: the morphisms are all canonically parametrized as a result of being sections. The inclusion is clearly unital. Hence what we have described is a topological category~$\S_f$.

\begin{definition}
The \emph{section category} of a continuous function~$f\colon X\rightarrow \R$ is the topological category~$\S_f$.
\end{definition}

Two continuous functions~$f\colon X\rightarrow \R$ and~$f'\colon X'\rightarrow \R$, together with a continuous function~$\phi \colon X \rightarrow X'$ over~$\R$ in the sense that~$f'\circ \phi = f$, induces a continuous functor~$\S_\phi\colon \S_f\rightarrow \S_{f'}$. So the assignment~$f\mapsto \S_f$ is functorial from the category of spaces over the real line.

Assume~$f\colon X\rightarrow \R$ to be a Reeb function from here on. Every section~$\sigma\colon [a,b]\rightarrow \R$ of~$f$ is decorated by two real numbers:~$f(s\sigma)=a$ and~$f(t\sigma)=b$. If~$A$ is a non-empty subset of~$\R$ containing the critical values of~$f$, we define the subcategory~$\C_f^A$ of sections decorated only by real numbers in~$A$:
\begin{definition}
\label{def:fullSubcat}
Let~$f\colon X\rightarrow \R$ be a Reeb function, and consider~$\mathrm{A}$ a non-empty subset of~$\R$ containing the critical values of~$f$. Define~$\C_f^A$ as the full subcategory of~$\S_f$ with object space~$\coprod_{a\subset A}f^{-1}a$.
\end{definition}

If~$A=\R$, then obviously~$\C_f^A=\S_f$. And more is true:~$\C_f^A$ and~$\S_f$ carries the same homotopical information for any choice of~$A$. We shall make this precise in Section~\ref{sec:mainresult}, after giving a brief recap on simplicial spaces.

\subsection{Some background on simplicial spaces}
\label{section:moreSimplicialStuff}

A \emph{simplicial set} is a family of sets~$X_n$,~$n\geq 0$, together with face maps~$d_i\colon X_n\rightarrow X_{n-1}$ and degeneracy maps~$s_j\colon X_n\rightarrow X_{n+1}$ satisfying certain relations~\cite[p.~4]{goerss2009simplicial}. It resembles a simplicial complex: the face map~$d_i$ applied to an~$n$--simplex is the~$(n-1)$--simplex to be interpreted as its~$i$th face. The degeneracy maps, on the other hand, encode the number of ways in which one could consider an~$n$--simplex as a degenerate~$(n+1)$--simplex. The latter is not important to us, for all homotopy types in this paper are unaffected by simply omitting degeneracy maps. This can be made precise by verifying the goodness condition in~\cite{segal1974categories}.\\

The \emph{nerve} is a functor
\[
\nerve\colon  \Cat\rightarrow \sSet.
\]
It maps a category~$\C$ to the simplicial set whose set of~$n$--simplices is the~$n$--fold pullback
\[(\nerve\C)_n=\mor\C\times_{\ob\C}\dots\times_{\ob\C}\mor\C,
\]
of composable~$n$--tuples of morphisms. The outer face maps~$d_0$ and~$d_n$ are given by omitting the first and last component, respectively, whereas the inner face maps~$d_i$ are given by composing the~$i$th and~$(i+1)$th component.

A \emph{simplicial space}~$X_\bullet$ is a simplicial set with the additional requirement that~$X_n$ is a topological space and the face and degeneracy maps are continuous. The nerve makes perfect sense as a functor
\[
\nerve\colon \topcat\rightarrow \sTop.
\] 
Denote by~$\Delta^\bullet$ the cosimplicial space with~$n$--cosimplices the standard topological~$n$--simplex~$\Delta^n$. The coface map~$d^i\colon \Delta^n \rightarrow \Delta^{n+1}$ is the inclusion of~$\Delta^{n+1}$'s~$i$th face, whereas the codegeneracy map~$s^j\colon \Delta^n\rightarrow \Delta^{n-1}$ collapses~$\Delta^n$ onto its~$j$th face. Then the \emph{geometrical realization}
\[
|\cdot|\colon \sTop\rightarrow \Top
\]
is defined by assigning to a simplicial space~$X_\bullet$ the quotient space
\[
|X_\bullet|=(\coprod\limits_n X_n \times \Delta^n  )/ \sim
\]
with relation generated by~$(d_ix,z)\sim (x,d^i z)$ and~$(s_j x,z)\sim (x,s^j z)$. Compose the realization with the nerve to define the~\emph{classifying space}
\[
\classifying=|\nerve(\cdot) |\colon \topcat\rightarrow \Top.
\]

\begin{example}
\label{example:nSimplexSf}
Let~$f\colon X\rightarrow \R$ be a continuous function and consider the associated simplicial space~$\nerve \S_f$ obtained from applying the nerve to the section category. For a sequence~$a_0\leq\dots\leq a_n$ we introduce the topological space
\[
\S_f[a_0,\dots,a_n]=\S_f[a_0,a_1]\times_{f^{-1}a_1}\S_f[a_1,a_2]\times_{f^{-1}a_2}\dots\times_{f^{-1}(a_{n-1})} \S_f[a_{n-1},a_n]
\]
of sections~$[a_0,a_n]\rightarrow X$ labeled by the given sequence.  With this notation, we may identify the space of~$n$--simplices
\[
(\nerve\S_f) = \coprod_{a_0\leq\dots\leq a_n \subset \R} \S_f[a_0,\dots,a_n].
\]

If, in addition, the function~$f\colon X\rightarrow \R$ is a Reeb function on a stratified space~$X$, then the simplicial space~$\nerve \C_f^A$ has an associated space of~$n$--simplices 
\[
(\nerve\C_f^A) = \coprod_{a_0\leq\dots\leq a_n\subset A} \S_f[a_0,\dots,a_n].
\]

\end{example}

\subsection{A proof of Theorem~\ref{intro:mainresult}}
\label{sec:mainresult}

Let~$f\colon X\rightarrow \R$ be a continuous function from a topological space~$X$ to the real line. It is tempting to presume~$X\simeq\classifying \S_f $ to be true in general (Theorem~\ref{intro:mainresult}). But that is not the case. 

\begin{example}
\label{example:BsfNotX}
There is a continuous function~$f\colon I\rightarrow \R$, from the unit interval~$I$, uniquely determined by the formula~$f(x)=x\sin(\frac{1}{x})$. Proposition ~\ref{prop:BSfZeroHomology}, to be proven, tells us that there cannot be a path from~$1$ to any other point in~$\classifying \S_f$: such a path would have to meet an infinite number of~$1$--cells up to homotopy fixing endpoints. Hence~$\classifying \S_f$ has at least two path components. In fact,~$\classifying \S_f$ is weakly equivalent to the disjoint union of two points--see Example~\ref{ex:oscillation}, to be computed.
\end{example}

Assume from here on that~$f\colon X\rightarrow \R$ is a Reeb function and that~$A$ is a subset of the real line satisfying the assumptions of Definition~\ref{def:fullSubcat}. Recall from Example~\ref{example:nSimplexSf} that the space of~$n$--simplices in~$\nerve\C_f^A$ is the disjoint union~$\coprod\S_f[\bar{a}]$, indexed over non-decreasing sequences~$\bar{a}=(a_0,\dots,a_n)$ in~$A$. Points in~$\classifying\C_f^A$ are thus classes~$[\bar{\sigma},\bar{t}]$ with~$(\bar{\sigma},\bar{t})$ in~$\S_f[\bar{a}]\times \Delta^n$. There is a map~$\phi\colon \classifying \C_f^A\rightarrow X$ which is soon to be proven a weak homotopy equivalence. For a representative with~$\bar{\sigma}=(\sigma_1,\dots,\sigma_n)$, a sequence of composable sections~$[a_{i-1},a_i]\rightarrow X$, and~$\bar{t}=(t_0,\dots,t_n)$, it is defined by~$\phi[\bar{\sigma},\bar{t}]=\sigma_n\circ\dots \circ \sigma_1 (\bar{a}\bar{t})$. The notation~$\bar{a}\bar{t}$ is short for the dot product~$a_0t_0+\dots+a_nt_n$. It is a straightforward hassle to verify that this does in fact induce a well-defined map on~$\classifying\C_f^A$. Post composition with~$f$ now defines a map~$\bar{f}=f\circ \phi$ from~$\classifying\C_f^A$ to the real line. Applying~$f$ to the above formula reveals that~$\bar{f}[\bar{\sigma},\bar{t}]=\bar{a}\bar{t}$ on representatives--for the composition~$\sigma_n\circ\dots \circ \sigma_1 $ in~$\C_f^A$ is a section of~$f$.

Let us establish some notation before proving some convenient lemmas. Given a finite non-decreasing sequence~$\bar{a}=(a_0,\dots,a_n)$ in~$A$ and a subspace~$K$ of~$\Delta^n$, we denote by~$\langle\bar{a}, K \rangle$ the image of
~$\coprod \S_f[\bar{b}]\times K$, ranging over all subsequences~$\bar{b}$ of~$\bar{a}$, under the quotient map that defines~$\classifying \C^A_f$. In particular,~$\langle \bar{a}, \Delta^n\rangle$ is the subspace generated by classes whose representative is decorated by a subsequence of~$\bar{a}$.

\begin{lemma}
\label{lemma:compactIntoCfA}
If~$m\colon K\rightarrow \classifying \C_f^A$ is a map from a compact space~$K$, then there is an increasing sequence~$\bar{a}$ in~$A$ of length~$n$ such that the image of~$K$ is contained in~$\langle \bar{a},\Delta^n\rangle$.
\end{lemma}
\begin{proof}
Let
\[
\sk_k \classifying \C_f^A =(\coprod\limits_{q\leq k} (\nerve\S_f)_q\times \Delta^q)/\sim
\]
be the~$k$--skeleton of~$\classifying \C_f^A$. It is well-known that the map~$m$ must factor through some~$k$--skeleton of~$\classifying\C_f^A$, because~$K$ is compact. In our notation, one may alternatively write~$\sk_k \classifying\C_f^A =\cup \langle\bar{b},\Delta^k \rangle$, ranging over all non-decreasing sequences~$\bar{b}=(b_0,\dots,b_k)$ in~$A$. Hence we can deduce even more: the image of~$m$ can only meet finitely many subspaces of the form~$\langle \bar{b} , \Delta^k \rangle$, i.e. it factorizes through~$\cup_{i=0,\dots, q} \langle\bar{b}_i,\Delta^m\rangle $ for finitely many sequences~$\bar{b}_i=(b_{i,0},\dots,b_{i,k})$. Include and order all the components~$b_{i,j}$ to define the bigger non-decreasing sequence~$\bar{a}=(a_0,\dots,a_n)$. A point~$[\bar{\sigma},\bar{t}]$ in the image of~$m$ comes with a subsequence of~$\bar{a}$.
\end{proof}

Recall the spine~$\mathrm{sp}\Delta^n$ of the topological~$n$--simplex. It is the subspace parametrized by tuples~$\bar{t}=(t_0,\dots,t_n)$ satisfying that at most two successive entries are non-zero; there is an~$i$ such that~$t_j=0$ for all~$j$ except possibly~$j=i-1,i$. Points, or classes, in~$\langle \bar{a}, \mathrm{sp} \Delta^n\rangle$ has a particularly nice presentation: a point~$[\bar{\sigma},\bar{t}]$ in~$\langle \bar{a}, \mathrm{sp} \Delta^n\rangle$ can be represented~$[\sigma_i,(t_{i-1},t_{i})]$, because of how~$\bar{t}=(0,\dots,t_{i-1},t_{i},0\dots,0)$ for some~$i$.

\begin{lemma}
\label{lemma:spine}
Consider a Reeb function~$f\colon X\rightarrow \R$ and~$\bar{a}$ a non-decreasing sequence in~$A$. The subspace~$\langle \bar{a}, \Delta^n \rangle $ deformation retracts onto~$\langle \bar{a}, \mathrm{sp}\Delta^n \rangle $ in~$\classifying \C_f^A$.
\end{lemma}
\begin{proof}
A point~$[\bar{\sigma},\bar{t}]=[\sigma_i,(t_{i-1},t_{i})]$ in~$\langle \bar{a}, \mathrm{sp} \Delta^n\rangle$ is mapped to~$t_{i-1}a_{i-1}+t_ia_i$ under~$\bar{f}$. For a fixed~$\bar{\sigma}$ in~$\S_f[\bar{a}]$ the map~$\bar{f}\circ(\bar{\sigma},\id_{\mathrm{sp}\Delta^n})\colon \mathrm{sp}\Delta^n\rightarrow \R$ is injective, because~$\bar{a}$ is an increasing sequence. So for every~$a_0\leq b\leq a_n$ and~$\bar{\sigma}$ there is a unique~$\bar{s}_b$ in~$\mathrm{sp}\Delta^n$ such that~$\bar{f}[\bar{\sigma},\bar{s}_b]=b$. Two points in~$\langle \bar{a}, \Delta^n \rangle \cap \bar{f}^{-1}b$ both associate to the same~$\bar{s}_b$. The homotopy
\[
R([\bar{\sigma},\bar{t}],t)=[\bar{\sigma},(1-t)\bar{t}+t\bar{s}_{\bar{f}[\bar{\sigma},\bar{t}]}]
\]
is thus well-defined. And it satisfies~$R(-,0)=\id_{\langle \bar{a}, \Delta^n \rangle }$ whereas the image of~$R(-,1)$ is contained in~$\langle \bar{a}, \mathrm{sp}\Delta^n \rangle $.  It is a deformation retract because~$\bar{s}_{\bar{f}[\bar{\sigma},\bar{t}]}=\bar{t}$ whenever~$\bar{t}$ is in the spine; the homotopy is trivial when restricted to the spine.
\end{proof}

The third lemma is analogous to Lemma~\ref{lemma:oneCriticalStratified}.

\begin{lemma}
\label{lemma:deformationRetract}
Consider a Reeb function~$f\colon X\rightarrow \R$ and~$\bar{a}=(a_0,\dots,a_n)$ an increasing subsequence of~$A$ such that~$[a_0,a_n]$ contains at most one critical value of~$f$. Then the subspace~$\langle \bar{a},\mathrm{sp}\Delta^n\rangle $ of~$\classifying\C_f^A$ deformation retracts onto
\begin{enumerate}[i)]
	\item $f^{-1}a$ for any~$a$ in~$\bar{a}$ if there is no critical value or
	\item $f^{-1}a$ for~$a$ equal to the critical value, otherwise. 
\end{enumerate}
\end{lemma}

\begin{proof}
We may assume~$a=a_0$, much like we only consider the case~$[a-n,a]$ in the proof of Lemma~\ref{lemma:oneCriticalStratified}: if~$a_i=a$, then we simply identify~$\langle \bar{a},\mathrm{sp}\Delta^n\rangle $ as the homotopy pushout of
\[
\langle (a_0,\dots,a),\mathrm{sp}\Delta^i\rangle\leftarrow f^{-1}a\rightarrow \langle (a,\dots,a_n),\mathrm{sp}\Delta^{n-i}\rangle
\]
and reduce to the case~$a=a_0$ or~$a=a_n$.

The deformation retract will be defined inductively in~$n$. If~$n=0$, then there is nothing to prove since~$\langle a,\Delta^0\rangle$ is equal to~$f^{-1}a$. Assume the existence of a deformation retract
\[
R_{n-1}\colon \langle (a_0,\dots, a_{n-1}), \mathrm{sp}\Delta^{n-1}  \rangle \times I \rightarrow \langle (a_0,\dots, a_{n-1}), \mathrm{sp}\Delta^{n-1}  \rangle
\]
onto~$f^{-1}a$. Then we need only compute a deformation retract of~$\langle \bar{a}, \mathrm{sp}\Delta^n \rangle $ onto~$\langle (a_0,\dots, a_{n-1}), \mathrm{sp}\Delta^{n-1}  \rangle$. But~$\langle \bar{a}, \mathrm{sp}\Delta^n \rangle $ is the union of~$\langle (a_0,\dots, a_{n-1}), \mathrm{sp}\Delta^{n-1}  \rangle$ and~$\langle (a_{n-1},a_n), \Delta^1\rangle$, so surely it suffices to deformation retract~$\langle (a_{n-1},a_n), \Delta^1\rangle$ onto~$f^{-1}a_{n-1}$.

A point~$[x]$ in~$f^{-1}(a_n)$, considered as a subspace of~$\langle (a_{n-1},a_n), \Delta^1\rangle$, has a canonical choice of representative~$[g_x, (0,1)]$ where~$g_x\colon [a_{n-1},a_n]\rightarrow X$ is a reparametrized flow-line provided by Proposition~\ref{proposition:modifiedFlowLines}. Hence every point in~$\langle (a_{n-1},a_n),\Delta^1\rangle\setminus f^{-1}a_{n-1} $ is represented~$[\sigma, (t_0,t_1)]$ with~$\sigma\colon [a_{n-1},a_n]\rightarrow X$ a section. To every section~$\sigma\colon [a_{n-1}, a_n]\rightarrow X$ we associate the map~$g_\sigma \colon [a_{n-1},a_n]^2\rightarrow X$ determined by~$g_\sigma (b,c)=g_{\sigma(b)}(c)$, where~$g_{\sigma (b)}$ is the reparametrized flow-line through~$\sigma(b)$. The section~$\sigma$ can be recovered from the composition~$[a_{n-1},a_n]\xrightarrow{d}[a_{n-1},a_n]^2\xrightarrow{g_\sigma} X$ where~$d$ is the diagonal map~$b\mapsto (b,b)$. For a fixed value~$b_0$, we also recover~$g_{\sigma(b_0)}$ from the composition~$[a_{n-1},a_n]\xrightarrow{(b_0,\id)} [a_{n-1},a_n]^2\xrightarrow{g_\sigma} X$. The straight line homotopy~$H_{b_0}\colon I\rightarrow \map ([a_{n-1,a_n}],[a_{n-1,a_n}]^2)$, 
\[
H_{b_0}(t)(b)=((1-t)b+tb_0, b)
\]
between~$d$ an~$(b_0,\id)$ will give us the desired deformation retract. Indeed, for a point~$[\sigma, (t_0,t_1)]$, contained in~$\langle (a_{n-1},a_n),\Delta^1\rangle\setminus f^{-1}a_{n-1}$, we declare
\[
R_n([\sigma, (t_0,t_1)],t)=[g_\sigma \circ H_{\bar{f}[\sigma, (t_0,t_1)]}(t), (1-t) (t_0,t_1) + t(1,0) ].
\]
When~$t=0$ the output is~$[\sigma, (t_0,t_1)]$ whereas~$t=1$ yields~$[g_{\phi[\sigma, (t_0,t_1)]}, (1,0)]=[g_{\phi[\sigma, (t_0,t_1)]} (a_{n-1})]$ in~$f^{-1} a_{n-1}$. Extend~$R_n$ to~$\langle(a_{n-1},a_n), \Delta^1\rangle$ by declaring~$R_n([x], t)=[x]$ for a class~$[x]$ in~$f^{-1}a_{n-1}$. 
\end{proof}

Combining Lemmas~\ref{lemma:spine} and~\ref{lemma:deformationRetract} imply the following.
\begin{lemma}
\label{lemma:refinement}
Assume that~$A'$ is constructed from~$A$ by adding a single real value. Then the inclusion~$\C_f^A\hookrightarrow \C_f^{A'}$ induces a weak homotopy equivalence on classifying spaces.
\end{lemma}

We are prepared to prove that the homotopy types of~$\classifying \C_f^A$ and~$X$ coincide for all eligible~$A$. In particular, Theorem~\ref{intro:mainresult} follows.

\begin{theorem}
\label{thm:mainresult}
Let~$f\colon X\rightarrow \R$ be a Reeb function. The map~$\phi \colon \classifying \C_f^A\rightarrow X$ is a weak homotopy equivalence for all non-empty subsets~$A\subset \R$ containing the critical values of~$f$.
\end{theorem}

\begin{proof}
We can essentially reduce the problem to two special cases: 1.~$f$ has no critical values and 2.~$f$ has precisely one critical value. Indeed, assume that there is at least two critical values. Enumerate them~$(c_i)$ according to the standard ordering of the real line. Every critical value~$c_i$ is contained in~$N_i=(c_{i-1},c_{i+1})$, possibly interpreting~$c_{i-1}=-\infty$ or~$c_{i+1}=\infty$. Whence we deduce open covers of~$X$ and~$\C_f^A$ by considering~$U_i= f^{-1}N_i$ and~$V_i=\bar{f}^{-1}N_i$, respectively. Do note that the only non-empty intersections are of the form~$N_{i,j}=N_i\cap N_j$ for~$j=i,i+1$. Hence it suffices to prove that~$\phi$ restricts to weak homotopy equivalences~$\phi|_{V_{i,j}}\colon V_{i,j}\rightarrow U_{i,j}$ with~$V_{i,j}=V_i\cap V_j$ and~$U_{i,j}=U_i\cap U_j$ subject to~$j=i,i+1$. This essentially follows since~$X$ and~$\classifying\C_f^A$ can be described as homotopy colimits over the given covers. See Theorem 6.7.9 in \cite{tom2008algebraic} for a precise reference.

The open cover~$N_i$ is constructed so that~$N_{i,j}=N_i\cap N_j$ contains precisely one critical value when~$i=j$, whereas it contains no critical values when~$j=i+1$. So the reduction to the above two special cases is made precise. 

Fix~$N=N_{i,j}$,~$U=U_{i,j}$ and~$V=V_{i,j}$ for~$j=i$ or~$i+1$. Also, pick a value~$a$ in~$N$. If~$i=j$, then~$a$ must be chosen as the critical value. Otherwise, it can be any real value. Lemma~\ref{lemma:refinement} allow us to assume that~$a$ is contained in~$A$ without loss of generality. We may thus assume the inclusion~$f^{-1}a\hookrightarrow U$ to factorize through~$\phi|_V$. By the two out of three property it only remains to see that~$f^{-1}a\hookrightarrow U$ and~$f^{-1}a\hookrightarrow V$ are weak homotopy equivalences. The first follows directly from Lemma~\ref{lemma:oneCriticalStratified}. 

To see that~$f^{-1}a\hookrightarrow V$ is a weak homotopy equivalence we consider an arbitrary commutative diagram
\begin{center}
		\begin{tikzcd}
			\partial\Delta^m  \arrow[d, "i"] \arrow[r, "m"]      & f^{-1}a \arrow[d] \\
			
			\Delta^m \arrow[r, "l"] & V
		\end{tikzcd}
\end{center}
and find a lift~$L\colon \Delta^m\rightarrow f^{-1}a$ up to homotopy relative to~$\partial\Delta^m$. Lemma~\ref{lemma:compactIntoCfA} tells us that the image of~$l$ is contained in a subspace~$\langle \bar{a}, \Delta^n\rangle $. The deformation retract onto~$\langle \bar{a}, \mathrm{sp}\Delta^n\rangle $, provided by Lemma~\ref{lemma:spine}, fixes~$f^{-1}a$ and hence~$l$ factorizes through~$\langle \bar{a}, \mathrm{sp}\Delta^n\rangle $, up to homotopy relative~$\partial \Delta^m$. A point in~$V$ must map into~$N$ under~$\bar{f}$, hence we can assume~$\bar{a}$ to be contained in~$N$ without loss of generality. Now Lemma~\ref{lemma:deformationRetract} applies to give a deformation retract from~$\langle \bar{a}, \mathrm{sp}\Delta^n\rangle $ to~$f^{-1}a$. But then we can homotope~$l$ to a map which factorizes through~$f^{-1}a$, up to homotopy relative to~$\partial \Delta^m$.
\end{proof}

\section{Reeb functions for homology computations}
\label{section:spectralsequence}
Section~\ref{subsec:sectionsequence} introduces the \emph{section spectral sequence} associated naturally to a continuous function~$f\colon X\rightarrow \R$. A noteworthy property of this spectral sequence is that its second page only consists of two non-trivial columns (Proposition~\ref{proposition:E2isInfty}). In Section~\ref{subsec:criticalsequence} we introduce the much smaller \emph{critical spectral sequence}. Basic computational tools are deduced and illustrated.

\subsection{The section spectral sequence}
\label{subsec:sectionsequence}
Let us fix a generalized homology theory~$k_\ast$. Recall that a simplicial space~$X_\bullet$ comes with a spectral sequence whose termination is~$k_\ast|X_\bullet|$. The cohomological version is derived in~\cite{segal1968classifying}. For every~$q$, we associate the simplicial abelian group~$k_q X_\bullet \colon \Delta^{\op} \rightarrow \Ab $ by applying~$k_q$ level-wise. Entries on the first page are given by~$\page^1_{p,q}=k_q X_p$, whereas the differential is induced by the face maps in~$k_q X_\bullet$. Indeed, a simplicial abelian group~$A_\bullet$ defines a chain complex by collapsing degeneracies entry-wise and defining the differential~$\partial=\sum_i (-1)^{i} d_i$. Entries on the second page are given by~$\page^2_{p,q}= \homology_p k_q X_\bullet$. A map~$F\colon X_\bullet\rightarrow Y_\bullet$ of simplicial spaces naturally induces homomorphisms~$k_qF_p\colon\page^1_{p,q}\rightarrow \bar{\page}^1_{p,q}$ and hence maps~$\page^2_{p,q}\rightarrow\bar{\page}^2_{p,q}$ on the second page. This does in fact define a homomorphism in the category of (homology) spectral sequences, but there is no need for further abstraction.

We have seen that a continuous function~$f\colon X\rightarrow \R$, on a topological space~$X$, defines a section category~$\S_f$ whose morphisms are the sections~$\sigma\colon [a,b]\rightarrow X$.
\begin{definition}
The \emph{section spectral sequence} of a continuous function~$f\colon X\rightarrow \R$ is the spectral sequence naturally associated to the simplicial topological space~$\nerve\S_f$.
\end{definition}

Theorem~\ref{intro:mainresult} tells us that the section spectral sequence converges to~$\homology_\ast X$ whenever~$f\colon X\rightarrow \R$ is a Reeb function:
\begin{proposition}
\label{prop:SS}
For~$f\colon X\rightarrow \R$ a Reeb function, the section spectral sequence converges to~$\homology_\ast X$:
\[
\homology_p k_q \nerve\S_f \Rightarrow k_{p+q} X.
\]
\end{proposition}

The first algebraic observation is concerned with computability: there are only two non-zero columns on the second page of the section spectral sequence. In particular, the sequence collapses on the second page; all differentials on the second page are zero.

\begin{proposition}
\label{proposition:E2isInfty}
Let~$f\colon X\rightarrow \R$ be a continuous function. The section spectral sequence of~$\nerve \S_f$ satisfies~$\page_{p,q}^2=0$ for~$p\geq 2$.
\end{proposition}
\begin{proof}
The additivity axiom of~$k_\ast$ implies
\[
E_{p,q}^1=k_q\Big(\coprod_{a_0< \dots < a_o} \S_f(a_0,\dots a_p)\Big)\simeq \bigoplus_{a_0< \dots < a_p} k_q\S_f(a_0,\dots,a_p).
\]
For an arbitrary~$q$ we fix a~$p\geq 2$ and denote by~$\partial_p\colon \page^1_{p,q}\rightarrow \page^1_{p-1,q}$ the boundary map. On the level of elements, an element~$\alpha$ in~$k_q\S_f[a_0,\dots,a_p]$ is mapped to the alternating sum~$\sum (-1)^{i} d_i \alpha$ with~$d_i\alpha$ an element of~$k_q\S_f[a_0,\dots,\hat{a}_i,\dots,a_p]$. We shall see that the kernel of~$\partial_p$ is contained in the image of~$\partial_{p+1}$, thus justifying the assertion.

An arbitrary element~$\alpha$ in the kernel of~$\partial_p$ can be presented as a linear combination~$c_1\alpha_1+\dots+c_n\alpha_n$ where each~$\alpha_i$ is in some~$k_q\S_f[a_{i,0},\dots,a_{i,p}]$. We proceed by defining a process which splits~$\alpha$ into a sum~$\alpha'+\partial_{p+1} \beta'$ and argue why finitely many iterations must eventually lead to~$\alpha=\partial_{p+1} \beta$ for some~$\beta$ in~$k_q(\nerve\S_f)_{p+1}$.

Start by picking the~$\alpha_i$ with smallest possible~$a_{i,0}$ and recursively smallest possible~$a_{i,j+1}$ subject to smallest possible~$a_{i,j}$. Since~$\alpha$ is assumed to be in the kernel of~$\partial_p$, there must be a~$j$ and~$n$ such that~$d_n\alpha_j$ cancels~$d_q\alpha_i$. But~$\alpha_i$ was chosen so that~$n$ must necessarily equal~$q$--otherwise we would have picked~$\alpha_j$ in place of~$\alpha_i$. There are now two cases:

1.The class~$\alpha_j$ restricts to~$\alpha_i$. That is, we can pick~$j$ such that there is a class~$\beta$ in~$k_q\S_f[a_{i,0},\dots,a_{i,p}, a_{j,p}]$ satisfying~$d_{p+1}\beta=\alpha_i$ and~$d_p \beta = \alpha_j$. If so, replace~$(-1)^p(\alpha_i-\alpha_j)$ with~$\partial_{p+1}\beta-\sum_{k\neq p,p+1}d_k\beta$ in the formula~$c_1\alpha_1+\dots+c_n\alpha_n$ to obtain a new formula~$\alpha=c_1'\alpha_1'+\dots+c_m'\alpha_m'+\partial_{p+1} \beta$. Now there is one less summand~$\alpha_k'$ with the minimal configuration of~$\alpha_i$.

2. Otherwise, if there is no such~$j$, we rather apply~$d_{p-1}$ to~$\alpha_i$. This produces an element~$d_{p-1}\alpha_i$ contained in~$k_q\S_f[a_{i,0},\dots,\hat{a}_{i,p-1},a_{i,p}]$. There must be a~$k$ and~$m$ such that~$d_m \alpha_k$ cancels~$d_{q-1}\alpha_i$. As in 1. the minimality of~$\alpha_i$ implies that~$m=p-1$. Since~$p\geq 2$ we can lift~$\alpha_i$ to~$k_q\S_f[a_{i,0},\dots,a_{i,p-1},a_{k,p-1},a_{i,p}]$ by adding the distinct label from~$\alpha_k$. Denote such a lift~$\beta$. This element satisfies~$d_{p-1}\beta=\alpha_k$ and~$d_p\beta = \alpha_i$. Moreover, among the~$d_l\beta$ for~$l\neq p-1,p$, only~$d_{q+1}\beta$ has a smaller configuration than~$\alpha_i$. As in~$1.$ we replace~$(-1)^p(\alpha_i-\alpha_k)$ with~$\partial_{p+1}\beta-\sum_{l\neq p-1,p} d_l\beta$, yielding a new linear combination~$\alpha=c_1'\alpha_1'+\dots+c_m'\alpha_m'+\partial_q \beta$. We did not manage to reduce the number of minimal configurations, but rather replaced a minimal configuration by a smaller minimal configuration.

Repeating this process iteratively must terminate. Indeed, case 2. can only repeat finitely many times as there are only finitely many summands in~$\alpha$. So we have successfully constructed an iterative process that reduces the number of summands with a minimal configuration after finitely many steps.
\end{proof}

With our choice of indexing, a homology spectral sequence with only two adjacent non-zero columns is a collection of short exact sequences. I refer to~\cite[p. 124]{weibel1995introduction} for details.
\begin{corollary}
\label{cor:SES}
Any continuous function~$f\colon X\rightarrow \R$ determines short exact sequences
\[
H_0 k_q (\nerve\S_f)_0\hookrightarrow k_q \classifying \S_f \twoheadrightarrow H_1 k_{q-1}(\nerve\S_f)_1
\]
for all~$q\geq 1$.
\end{corollary}

An~$f$--path is a path~$p\colon I\rightarrow X$ in~$X$ which reparametrizes to a concatenation of paths contained in fibers of~$f$ and possibly reversed sections of~$f$. Define the relation~$\sim_f$ on~$X$ by declaring that points are equivalent if they can be connected by an~$f$--path.
\begin{proposition}
\label{prop:BSfZeroHomology}
For~$f\colon X\rightarrow \R$ a continuous function, the abelian group~$\homology_0\classifying \S_f$ is the free abelian group on~$X/ \sim_f$.
\end{proposition}
\begin{proof}
The assertion is a direct consequence of~$\homology_0 \classifying \S_f=\page^2_{0,0}$. Indeed,~$\page^2_{0,0}=\homology_0\homology_0 \nerve \S_f$ where the first application of~$\homology_0$ identifies points in the same path components of fibers, whereas the second identifies points that can be connected via sections.
\end{proof}

The following calculation justifies Example~\ref{example:BsfNotX}.
\begin{example}
\label{ex:oscillation}
Consider the continuous function~$f\colon I\rightarrow \R$,~$f(x)=x\sin (\frac{1}{x})$ . We shall verify that the homology of~$\classifying \S_f$ with coefficients in~$\Z$ is
\[
\homology_n\classifying \S_f = \twopartdef{\Z^2}{n=0}{0}{n\neq 0}
\]

It follows from Proposition~\ref{proposition:E2isInfty} that the associated spectral sequence has~$\page^2_{p,q}=0$ whenever~$p\geq 2$. Moreover, all of  the section spaces~$\S_f[a_0,\dots,a_n]$ are discrete except for~$\S_f[0]=f^{-1}0$ which is homeomorphic to the subspace~$\{\frac{1}{n}|n=1,2,\dots \}\cup 0$ of~$\R$. Nonetheless, the topological space~$\S_f[0]$ is homotopically discrete as it is weakly equivalent to the natural numbers equipped with the discrete topology. We thus conclude that~$\page^1_{p,q}=0$ for~$q\geq 1$. So in light of Proposition~\ref{proposition:E2isInfty} it only remains to calculate~$\page^2_{0,0}$ and~$\page^2_{1,0}$.

Let~$\partial_1\colon \page^1_{1,0}\rightarrow \page^1_{0,0}$ be the remaining non-zero differential on the first page and pick a linear combination~$c_1\sigma_1+\dots+c_n\sigma_n$ in its kernel. Let~$m_i$ denote the minimum of~$\sigma_i$, a section into~$I$, and consider~$m$ the smallest number among the~$m_i$. Observe how every~$\sigma_i$ maps into~$I_m=[m,1]$. The inclusion~$j\colon I_m\hookrightarrow I$ satisfies~$f|_{I_m}=f\circ j$ and so there is an induced map~$\bar{\page}^\ast_{p,q}\rightarrow \page^\ast_{p,q}$, between associated spectral sequences. Theorem~\ref{intro:mainresult} applies to~$f_{I_m}$ so that~$\homology_1 I_m=0$ implies~$\bar{\page}^2_{0,1}=0$. Hence~$c_1\sigma_1+\dots+c_n\sigma_n$ is in the image of~$\bar{\partial}_1\colon \bar{\page}^1_{2,0}\rightarrow \bar{\page}^1_{1,0}$, but then it is also in the image of~$\partial_2\colon \page^1_{2,0}\rightarrow \page^1_{1,0}$. The isomorphism~$\page^2_{1,0}\simeq0$ thus follows.

Proposition~\ref{prop:BSfZeroHomology} implies~$\page^2_{0,0}\simeq \Z^2$: All pairs in~$(0,1]$ can be connected by a finite zigzag of sections, whereas no section starts nor terminates in~$0$. Hence~$\homology_0\classifying\S_f\simeq \Z^2$.
\end{example}
\subsection{The critical spectral sequence}
\label{subsec:criticalsequence}
For a continuous function~$f\colon X\rightarrow \R$, the associated section spectral sequence
\[
\homology_qk_p \nerve \S_f \Rightarrow \homology_{p+q} \classifying \S_f
\]
collapses on the second page which only consists of two non-zero columns. Alas, it has a huge problem when it comes to computability. Computing~$\homology_p k_q \nerve \S_f$ abouts to an uncountable number of homology computations. For we would have to determine~$k_q\S_f[a,b]$ for all real numbers~$a\leq b$. But if~$f\colon X \rightarrow \R$ is a Reeb function, we shall see that the complexity is drastically reduced.

Recall from Section~\ref{subsec:sf} that for every non-empty subset~$A$ of~$\R$ which contains the critical values of~$f$, there is the subcategory~$\C_f^A$ of sections whose source and target are both contained in~$A$. From here on we assume~$f$ to have at least one critical value and introduce the \emph{critical category}~$\C_f=\C_f^{\{ \text{critical values}\}}$. 
\begin{definition}
\label{def:criticalSS}
The \emph{critical spectral sequence} of a Reeb function~$f\colon X \rightarrow \R$ is the spectral sequence naturally associated to~$\nerve\C_f$.
\end{definition}

Theorem~\ref{thm:mainresult} tells us that the critical spectral sequence converges to~$\homology_\ast X$.
\begin{proposition}
\label{prop:computability}
For~$f\colon X \rightarrow \R$ a Reeb function, the critical spectral sequence converges to~$k_\ast X$: 
\[
\homology_qk_p \nerve \C_f \Rightarrow k_{p+q} X.
\]
\end{proposition}

As opposed to the section spectral sequence, we need only compute the generalized homology groups of~$\S_f[c_0,c_1]$ whenever~$c_0<c_1$ are critical values of~$f$. If for example~$X$ is compact, this reduces the number of~$q$th generalized homology groups needed to compute from uncountable to finite. We shall reduce the complexity even further: it suffices to compute the generalized homology groups of~$\S_f[c_0,c_1]$ whenever~$c_0<c_1$ are \emph{successive} critical values. Let us introduce some convenient notation before stating the formal result.

I remind that~$\page^1_{p,q}\simeq \oplus k_q\S_f[c_0,c_1]$ ranging over all critical values~$c_0< c_1$. For every~$q\geq 0$, the differential~$\partial^1_{1,q}\colon  \page^1_{1,q}\rightarrow \page^1_{0,q}$ restricts to a morphism
\[
\partial^s_{1,q}\colon \bigoplus k_q\S_f[c_0,c_1] \rightarrow \bigoplus k_q\S_f[c].
\]
only ranging over successive pairs of critical values~$c_0<c_1$. 
\begin{proposition}
\label{prop:easyToCompute}
For~$f\colon X \rightarrow \R$ a Reeb function, the second page of the associated critical spectral sequence satisfies~$\page^2_{1,q}\simeq \ker \partial^s_{1,q}$
for all~$q\geq 2$.
\end{proposition}

\begin{proof}
Let~$\beta$ be an element in~$\page^1_{2,q}$. The relation imposed by~$\partial^1_{2,q}$ on~$\page^2_{1,q}$ is determined by the equation~$\partial^1_{2,q}\beta=d_0\beta-d_1\beta+d_2\beta$ and implies that~$[d_1\beta]=[d_0\beta]+[d_2\beta]$. If~$\beta$ is in~$k_q\S_f[c_0,c_1,c_2]$, then~$d_0\beta$ is in~$k_q \S_f[c_1,c_2]$ etc. So the map~$d_i$ simply forgets the~$i$th label. 

For a~$[\alpha]$ in~$\page^2_{1,q}$ represented by~$\alpha$ in~$k_q \S_f[c_0,c_1]$, we list all intermediate critical values~$d_0<\dots < d_n$ with~$d_0=c_0$ and~$d_n=c_1$. Now it is only a matter of applying the relation imposed by~$\partial^1_{2,q}$~$n$ times to rewrite~$[\alpha]$ as a linear combination~$[\alpha_1]+\dots+[\alpha_n]$ with~$\alpha_i$ in~$k_q \S_f[d_{i-1},d_i]$. Hence the elements in~$\ker \partial^s_{1,q}$ generates~$\page^2_{1,q}$. Moreover, the relation induced by~$\partial^1_{2,q}$ is trivial on this set of generators, for they cannot be decomposed further.
\end{proof}

As a last computational tool, we recognize the homotopy type of section spaces decorated by successive critical values. These section spaces has the homotopy type of any intermediate fiber. Hence the critical sequence recovers the homology of~$X$ using the homology type of certain fibers. I refer to Example~\ref{example:torus} for a hands-on demonstration.
\begin{proposition}
\label{prop:finite}
Consider~$c$ and~$d$ two successive critical values of~$f$ as well as a real number~$a$ in~$(c,d)$. The evaluation map~$\ev_a\colon\S_f[c,d]\rightarrow f^{-1}a$ is a homotopy equivalence.
\end{proposition}
\begin{proof}
Let~$g\colon [c,d]\times f^{-1}(c,d)\rightarrow X$ be a family of reparametrized sections which exists per Proposition~\ref{proposition:modifiedFlowLines}. Its adjoint~$\bar{g}\colon f^{-1}(c,d)\rightarrow \S_f[c,d]$ restricts to a map~$g_a \colon f^{-1}a\rightarrow \S_f[c,d]$, mapping a point~$x$ in~$f^{-1}a$ to the reparametrized flow-line $g_x$ through~$x$. It is clear that~$\ev_a\circ g_a$ is the identity on~$f^{-1}a$. Conversely, the composition~$ g_a\circ\ev_a$ maps a section~$\sigma\colon [c,d]\rightarrow X$ to the reparametrized flow-line~$g_{\sigma(a)}$ through~$\sigma(a)$. Note that we may identify~$\sigma$ with the section~$b\mapsto g_{\sigma(b)}(b)$.

Denote by~$H\colon~[c,d]\times I\rightarrow [c,d]$ the straight line homotopy~$H(b,t)=(1-t)b+ta$, from which we define a homotopy~$G\colon \S_f[c,d]\times I\rightarrow\S_f[c,d]$. It maps a tuple~$(\sigma,t)$ to the section~$b\mapsto g_{\sigma \circ H(b,t)}(b)$. Notice that~$G(-,0)$ is~$\sigma$ whereas~$G(-,1)$ is~$g_{\sigma(a)}$. We have thus constructed a homotopy from the identity on~$\S_f[c,d]$ to~$ g_a\circ\ev_a$.
\end{proof}

We calculate the torus' homology as means to illustrate the computational implications of Propositions~\ref{prop:easyToCompute} and~\ref{prop:finite}.

\begin{example}
\label{example:torus}
Let~$h\colon T\rightarrow \R$ be the height function on the torus depicted

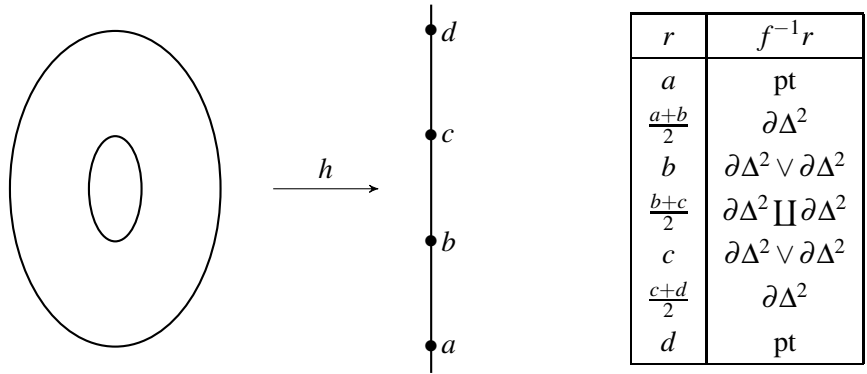
\begin{center}
\begin{tikzpicture}[scale= 0.7]
    \draw [thick] (0,0) ellipse (2cm and 3cm);
    \draw [thick] (0,0) ellipse (0.5cm and 1cm);
	\draw[->] (3,0)--(5,0);
	\draw [thick] (6,-3.5) -- (6,3.5);

	\node [] at (6,1) {$\bullet$};
	\node [] at (6,-1) {$\bullet$};
	\node [] at (6,3) {$\bullet$};
	\node [] at (6,-3) {$\bullet$};

	\node [right] at (6,1) {$c$};
	\node [right] at (6,-1) {$b$};
	\node [right] at (6,3) {$d$};
	\node [right] at (6,-3) {$a$};
	
	\node [above] at (4, 0) {$h$};

	\node [] at (12,0) {
 \begin{tabular}{ |c | c| } 
 \hline
 $r$ & $f^{-1} r$  \\ 
 \hline
 $a$ & pt  \\ 
 $\frac{a+b}{2}$ & $\partial \Delta^2$  \\ 
 $b$ & $\partial\Delta^2 \vee \partial \Delta^2$  \\ 
 $\frac{b+c}{2}$ & $\partial \Delta^2 \coprod \partial \Delta^2$  \\ 
 $c$ & $\partial\Delta^2 \vee \partial \Delta^2$  \\ 
 $\frac{c+d}{2}$ & $\partial \Delta^2$  \\ 
 $d$ & pt   \\ 
 \hline
\end{tabular}
};
\end{tikzpicture}
\end{center}

It has four critical values~$a$,~$b$,~$c$ and~$d$. Proposition~\ref{prop:finite} allows us to fill out the above table of homotopy types with ease. We determine the first page of the critical spectral sequence, knowing that we need only compute the homology of section spaces between successive critical values (Proposition~\ref{prop:easyToCompute}):

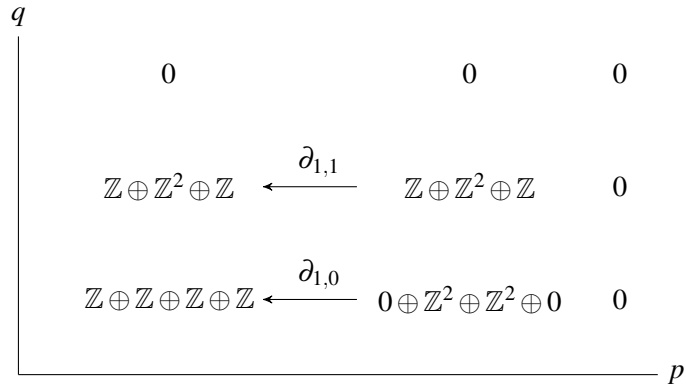
\begin{center}
\begin{tikzpicture}
	\draw (0,0) -- (8.5,0);
	\draw (0,0) -- (0,4.5);
	\draw [->] (4.5,1) -- (3.25,1);
	\draw [->] (4.5,2.5) -- (3.25,2.5);
	
	\node [] at (2,1) {$\Z\oplus \Z\oplus \Z \oplus \Z$};
	\node [] at (2,2.5) {$\Z\oplus \Z^2\oplus \Z$};
	\node [] at (2,4) {$0$};

	\node [] at (6,1) {$0\oplus \Z^2\oplus \Z^2\oplus 0$};
	\node [] at (6,2.5) {$\Z\oplus \Z^2 \oplus \Z$};
	\node [] at (6,4) {$0$};

	\node [] at (8,1) {$0$};
	\node [] at (8,2.5) {$0$};
	\node [] at (8,4) {$0$};

	\node [right] at (8.5,0) {$p$};
	\node [above] at (0,4.5) {$q$};

	\node [above] at (4,1) {$\partial_{1,0}$};
	\node [above] at (4,2.5) {$\partial_{1,1}$};
\end{tikzpicture}
\end{center}
Homology groups are split up according to the above table, e.g.
\[
\page^1_{0,1}=\homology_0 \partial \Delta^1\oplus \homology_0(\partial \Delta^1\coprod \partial \Delta^1)\oplus\homology_0 \partial \Delta^1\simeq \Z\oplus \Z^2 \oplus \Z.
\]
The differentials are induced by subtracting target from source:~$\partial=d_0-d_1=t-s$. For instance, the induced map~$\homology_0 t\colon \homology_0\frac{a+b}{2}\rightarrow \homology_0 f^{-1} b$ is the identity~$1\colon \Z \rightarrow \Z$ in coordinates. This is because the target of a flow-line through~$f^{-1}\frac{a+b}{2}$ meets the path component of~$f^{-1} b$. By such geometric reasoning we deduce
\[
\partial_{1,0}=
\begin{bmatrix}
    -1 & 0  & 0  & 0  \\
    1  & -1 & -1 & 0  \\
    0  & 1  &  1 & -1 \\
    0  & 0  &  0 & 1
\end{bmatrix}
\;\; \text{ and } \;\;
\partial_{1,1}=
\begin{bmatrix}
    1 & -1 & 0  & 0  \\
    1 & 0  & -1 & 0  \\
    0 & 1  &  0 & -1 \\
    0 & 0  &  1 & -1
\end{bmatrix}.
\]
Elementary linear algebra gives the second page:
\begin{center}
\begin{tikzpicture}
	\draw (0,0) -- (8.5,0);
	\draw (0,0) -- (0,4.5);
	
	\node [] at (2,1) {$\Z$};
	\node [] at (2,2.5) {$\Z$};
	\node [] at (2,4) {$0$};
	\node [] at (5,1) {$\Z$};
	\node [] at (5,2.5) {$\Z$};
	\node [] at (5,4) {$0$};
	\node [] at (8,1) {$0$};
	\node [] at (8,2.5) {$0$};
	\node [] at (8,4) {$0$};

	\node [right] at (8.5,0) {$p$};
	\node [above] at (0,4.5) {$q$};
\end{tikzpicture}
\end{center}
And so we read of that~$\homology_0 T \simeq \Z$,~$\homology_2 T\simeq\Z$ whereas~$\homology_1 T$ is an extension of~$\Z$ by~$\Z$ hence~$\Z^2$. The remaining homology groups are trivial.
\end{example}

\section{Reeb spaces}
\label{section:reeb}

The combinatorial Reeb space is introduced in Section~\ref{subsec:comb}. Section~\ref{section:thmA} is merely a recap of Quillen's theorem~A and the theory of collapsing schemes due to K. Brown. The last two Sections are dedicated to clarifying and proving Theorems~\ref{intro:combinatorialReebIsClassicalReeb} and~\ref{intro:combinatorialReebIsGraph}.

\subsection{From topological to combinatorial Reeb spaces}
\label{subsec:comb}
The topological Reeb space is defined for any continuous function~$f\colon X\rightarrow \R$. Given such a function one declares points to be equivalent if they are in the same path component of some fiber:~$x\sim_f y$ if there is a real number~$a$ and~$[x]=[y]$ in~$f^{-1}(a)$. 
\begin{definition}
\label{definition:topologicalReebSpace}
The \emph{topological Reeb space} associated to a continuous function~$f\colon X\rightarrow \R$ is the quotient space~$\reeb_f=X/\sim_f$.
\end{definition}
The topological Reeb space is commonly referred to as the Reeb graph, which surely is accurate for e.g. Morse functions, albeit not accurate in general.

\begin{example}
\label{example:ReebNotGraph}
Consider the Hawaiian earring~$\mathbb{H}$ embedded as a subspace of~$\mathbb{R}^2$:
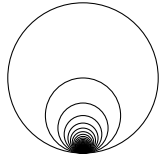
\begin{center}
\begin{tikzpicture}[scale=1]
\foreach \n in {1,...,1000} 
{
\draw (0,1/\n) circle (1/\n);
}
\end{tikzpicture}
\end{center}
The fibers of the horizontal projection~$\pr_1\colon \mathbb{H}\rightarrow \R$ are all discrete. We thus conclude that the topological Reeb space~$\reeb_{\pr_1}$ is homeomorphic to~$\mathbb{H}$. But the fundamental group of~$\mathbb{H}$ is not free~\cite{de1992fundamental}.
\end{example}

I, for one, would very much like to define a Reeb space whose homotopy type is that of a graph. This could very well serve as motivation for our next definition. Recall that a continuous function~$f\colon X\rightarrow \R$ has an associated section category~$\S_f$. Applying the nerve followed by the level-wise path components functor produces a simplicial set~$\pi_0\nerve\S_f$.
\begin{definition}
\label{definition:combinatorialReeb}
The \emph{combinatorial Reeb space} of a continuous function~$f\colon X \rightarrow \R$ is the simplicial set~$\pi_0\nerve \S_f$.
\end{definition}
Simplicial sets carry a homotopy theory equivalent to the standard theory on topological spaces: The homotopy type of a simplicial set~$S$ is equivalent to that of the topological space~$|S|$. In particular, the homotopy types of~$\pi_0\nerve\S_f$ and~$\reeb_f$ can be compared by realizing~$\pi_0\nerve\S_f$.

\subsection{More background on simplicial sets}
\label{section:thmA}

I will give a brief reminder on Quillen's well-known theorem A~\cite{quillen1973higher} as well as a theorem on collapsing schemes due to K. Brown~\cite{brown1992geometry}. Both are useful to prove Theorems~\ref{intro:combinatorialReebIsClassicalReeb} and~$\ref{intro:combinatorialReebIsGraph}$.

Let us first review Quillen's theorem A. For any functor~$F\colon\C\rightarrow\D$ and object~$d$ in~$\D$, we define the slice category~$F\downarrow d$ as the pullback
\begin{center}
\begin{tikzcd}
 F\downarrow y  \arrow[d] \arrow[r]  &  \Fun([1],\D) \arrow[d, "{(\ev_0,\ev_1)} "] \\
 \C\times[0] \arrow[r, "F\times d"] 				     &  \D\times \D
\end{tikzcd}
\end{center}
where~$\Fun([1],\D)$ is the category of functors~$[1]\rightarrow \D$. More explicitly, an object is a tuple~$(c,m)$ in the product~$\ob\C\times \mor\D$ subject to~$s(m)=F(c)$ and~$t(m)=d$; a morphism~$\alpha\colon (c,m)\rightarrow (c',m')$ is a morphism~$\alpha\colon c\rightarrow c'$ such that~$m= m'\circ F(\alpha)$. This data is commonly depicted

\begin{center}
\begin{tikzcd}
F(c)  \arrow[d, "m"] & \\
d             &
\end{tikzcd}
  and 
\begin{tikzcd}
& F(c)  \arrow[dr, "m"] \arrow[rr, "F(\alpha)"]    &   & F(c') \arrow[dl, "m' "] \\
&  							 & d &
\end{tikzcd}
\end{center}

Quillen's Theorem A gives a sufficient condition as to when~$F$ realizes to a weak homotopy equivalence: If~$\classifying (F\downarrow d)$ is contractible for all~$d$, then~$\classifying F$ is a weak homotopy equivalence.

For~$S$ a simplicial set, the topological space~$|S|$ is a CW complex whose cells are in bijection with the non-degenerate simplices in~$S$. In particular, it makes sense to talk about~$d_i e$, the~$i$th face of a cell~$e$ in~$|S|$. Moreover, we define the~$i$th horn of~$e$, which we will denote~$e_i$, to be the union of all its faces except the~$i$th one. We may safely deform~$|S|$ onto a quotient space~$Y$ by collapsing~$e$ onto~$e_i$ without changing the homotopy type of~$S$. Moreover,~$Y$ is clearly a CW complex again. Brown gives conditions for how to iterate this process of collapsing cells without changing the homotopy type of~$|S|$, while making sure that~$Y$ is still a CW complex.

Partition the non-degenerate simplices of~$S$ into three classes: essential, redundant and collapsible. The cells corresponding to redundant simplices are to be deformed along the collapsible cells, hence they are truly "redundant". So a function~$c$ from redundant simplices to collapsible simplices that maps~$n$--simplices to~$(n+1)$--simplices is required. If~$s$ is redundant and~$cs$ admits another redundant face~$s'$, then we write~$s'\leq s$. This data defines a \emph{collapsing scheme} if i)~$c$ is a bijection from redundant~$n$--simplices to collapsible~$(n+1)$--simplices for all~$n$ and ii) there is no infinite descending chain~$s \geq s'\geq s''\geq\cdots$.

Proposition 1 in~\cite{brown1992geometry} can then be formulated: For a collapsing scheme on~$S$, the quotient map~$|S|\rightarrow Y$ is a weak homotopy equivalence onto a CW complex~$Y$ whose~$n$--cells are in bijection with the essential simplices in~$S$.

\subsection{Proof of Theorem~\ref{intro:combinatorialReebIsClassicalReeb}}
\label{subsec:combistop}
The nerve admits a left adjoint~$\tau_1\colon \sSet\rightarrow \Cat$ commonly referred to as the \emph{fundamental category}. It agrees with the homotopy category when restricted to quasi-categories. A simplicial set~$S$ is sent to the category~$\tau_1 S$ whose object set is~$S_0$; morphism set is the \emph{directed paths}~$S_1\coprod (S_1\times_{S_0} S_1) \coprod\cdots$ modulo the relations~$s_0 x\sim 1_x$ for all~$0$--simplices~$x$ and~$d_1s\sim d_0s \circ d_2s$ for all~$2$--simplices~$s$. More explicitly, a directed path is a tuple~$(e_1,\dots,e_n)$ of edges/$1$--simplices such that the source of~$e_{i+1}$ is the target of~$e_i$;~$d_1e_{i+1}= d_0 e_i$. The corresponding morphism in~$\tau_1$ is denoted~$e_1\cdots e_n$, utilizing the word notation. We define the~\emph{length} of a word~$e_1\cdots e_n$ to be~$n$ if there is no equivalent word on fewer letters. An~$n$--simplex in~$\nerve\tau_1 S$ is a tuple~$(w_1,\dots,w_n)$ of composable words/morphisms. We define the length of~$(w_1,\dots,w_n)$ to be the length of the word~$w_1\cdots w_n$.

Since~$\tau_1$ is left adjoint to~$\nerve$, there is an associated unit map~$\eta\colon S\rightarrow \nerve \tau_1 S$, which is natural in~$S$. It is not a weak homotopy equivalence in general, as pointed out by Thomason in~\cite{thomason1980cat}. But we shall see that the unit always induces a weak homotopy equivalence on combinatorial Reeb spaces. 

\begin{lemma}
\label{lemma:abstract}
Let~$S$ be a simplicial set such that
\begin{enumerate}[i)]
	\item $S$ and~$\nerve \tau_1 S$ only differ in cells of length~$\geq 2$ and
	\item any word of length~$n$ has a unique presentation on~$n$ words.
\end{enumerate}
Then the natural map~$\eta\colon S\rightarrow \nerve\tau_1 S$ is a weak homotopy equivalence.
\end{lemma}
\begin{proof}
The theory of collapsing schemes, due to K. Brown~\cite{brown1992geometry}, is utilized to construct a homotopy inverse of the realized unit~$|\eta|$. I refer to Section~\ref{section:thmA} for a quick summary of this theory. All morphisms in~$\tau_1S$ will be represented uniquely per assumption ii).

To partition~$\nerve\tau_1 S$ into redundant, collapsible and essential simplices we first declare every~$1$--simplex~$e_1\cdots e_n$ of length~$n\geq 2$ redundant and define its associated collapsible~$2$--simplex
\[
c(e_1\cdots e_n )=(e_1\cdots e_{n-1} , e_n).
\]
This function is well-defined  because such a presentation is unique. For~$m\geq 2$ we declare an~$m$--simplex of the form~$(e_{1,1} \cdots e_{1,i_1},\dots,e_{m,1} \cdots e_{m,i_m})$, whose length is greater than or equal to~$2$, redundant if its not in the image of~$c$. Its associated~$(n+1)$--simplex is then determined by taking the biggest~$k$ such that~$i_k\geq 2$ and factoring~$e_{k,1}\cdots e_{k,i_k}$ as~$(e_{k,1}\cdots e_{k,i_k-1}, e_{k,i_k})$. In other words, we factorize out the last letter not already factorized. The remaining simplices are declared essential. Do note that these are precisely the ones whose length is equal to~$1$.

The function~$c$ is constructed to be a bijection in the sense required by a collapsing scheme. Whereas the second demand follows since a chain associated to a redundant~$n$--simplex~$s$ cannot exceed the length of~$s$ and is therefore necessarily bounded. We thus have a map~$|S|\rightarrow Y$ with~$Y$ a CW complex whose~$n$--cells correspond to essential~$n$--simplices, i.e. those of length~1. Hence~$Y$ is necessarily equal to~$|S|$ and what we have constructed is a homotopy inverse to~$|\eta|$.
\end{proof}

For a combinatorial Reeb space~$\pi_0\nerve\S_f$, a morphism in~$\tau_1\pi_0 \nerve \S_f$ is a word~$[\sigma_1][\sigma_2]\dots,[\sigma_n]$ with representatives~$\sigma_i$ in~$\S_f[a_{i-1},a_i]$ subject to~$[s\sigma_{i+1}]=[t\sigma_{i}]$. So the source and target of successive classes must agree up to path components in fibers. 

\begin{lemma}
\label{lemma:dipath}
A word in~$\tau_1\pi_0\nerve\S_f$ of length~$n$ has a unique presentation on~$n$ letters.
\end{lemma}
\begin{proof}
Assume that a word~$w$ is presented~$[\sigma_1]\cdots[\sigma_n]$ and~$[\rho_1]\cdots[\rho_n]$. We shall see that for all~$i$ the equality~$[\sigma_i]=[\rho_i]$ holds. From the domains of~$\sigma_i$ and~$\rho_j$ we extract sequences~$\bar{a}=(a_0,\dots,a_n)$ and~$\bar{b}=(b_0,\dots,b_n)$ of real numbers. 

These sequences must be equal. The statement is clear if~$n=1$ since~$a_0=b_0$ and~$a_n=b_n$. If~$n \geq 2$ we conversely assume that~$\bar{a}\neq \bar{b}$. Consider~$i$ the smallest index such that~$a_i\neq b_i$. We assume~$a_i<b_i$ without loss of generality. Introduce the real number~$c=\mathrm{min}(a_{i+1},b_i)$ and define two letters~$[\sigma']=[\sigma_{i+1}|_{[a_i,c]}]$ and~$[\rho']=[\rho_i|_{[b_{i-1},c]}]$ to rewrite~$[\sigma_i][\sigma']=[\rho']$. We observe that~$[\sigma_{i+1}]$ and~$[\rho']$ overlap in~$[\sigma']$ and so there must be a section~$\tau$ satisfying~$[\tau]=[\sigma_i][\sigma_{i+1}]$, contradicting the length of~$w$.

Hence the equality~$[\sigma_1]\cdots[\sigma_n]=[\rho_1]\cdots[\rho_n]$ can only be achieved if the letters are equal. For the only relation connecting them is to concatenate sections up to paths in fibers.
\end{proof}

The simplicial sets~$\pi_0\nerve\S_f$ and~$\nerve\tau_1\pi_0\nerve\S_f$ clearly only differ by simplices of length~$\geq 2$: a word~$(w_1,\dots,w_n)$ can only be of length~$1$ if~$w_i=[\sigma_i]$ and there is a section~$\rho$ such that~$[\rho]=[\sigma_1]\cdots[\sigma_n]$. 
\begin{lemma}
\label{lemma:addComposition}
Let~$f\colon X\rightarrow \R$ be a continuous function. The unit~$\eta\colon \pi_0\nerve\S_f\rightarrow \nerve\tau_1 \pi_0\nerve\S_f$ realizes to a weak homotopy equivalence.
\end{lemma}
\begin{proof}
A direct consequence of Lemmas~\ref{lemma:abstract} and~\ref{lemma:dipath}.
\end{proof}
In the proof of Theorem~\ref{intro:combinatorialReebIsClassicalReeb} it will be convenient to utilize Theorem~\ref{intro:mainresult}. But to do so, we must first verify that a Reeb function~$f\colon X\rightarrow \R$ defines a Reeb function on the topological space~$\mathrm{R}_f$. 
\begin{lemma}
\label{lemma:ReebhasReebData}
If~$f\colon X\rightarrow \R$ is a Reeb function, then~$\mathrm{R}_f$ has the homeomorphism type of a~$1$--dimensional CW complex satisfying that the induced function~$\bar{f} \colon \reeb_f \rightarrow \R$ is piecewise linear. 
\end{lemma}
\begin{proof}
Let~$\mathrm{R}_f$ be the topological Reeb space, presented as a quotient space according to Definition~\ref{definition:topologicalReebSpace}. A topological space~$Q$, homeomorphic to~$\mathrm{R}_f$, will be constructed to satisfy the assertion. We assume, without loss of generality, that every point in~$X$ is either contained in some critical level or in between two critical levels. This can always be achieved by slightly modifying the stratification on~$X$: one may e.g. present~$X$ as a filtered colimit of preimages of closed intervals under the map~$f$.

The set of~$0$--cells is given by~$\coprod \pi_0 f^{-1}c$ ranging over all critical values~$c$, whereas the set of~$1$--cells is given by~$\coprod \pi_0 \S_f[c,d]$ ranging over all successive critical values~$c<d$. The attaching maps comes from the source and target: a~$1$--cell~$e$ labeled by a path component~$[\sigma]$ in~$\pi_0\S_f[c,d]$ admits a source in~$\pi_0 f^{-1}c$;~target in~$\pi_0f^{-1}d$. Denote the resulting CW complex~$Q$. 

Define the piecewise linear map~$\bar{f}\colon Q\rightarrow \R$ as follows. On a closed~$1$--cell~$e\simeq [0,1]$ labeled by a class in~$\pi_0 \S_f[c,d]$ it is the orientation-preserving linear map~$[0,1]\rightarrow [c,d]$. Note that~$\bar{f}\colon Q\rightarrow \R$ is constructed to be piecewise linear.

There is a rather evident surjective map~$q\colon X\rightarrow Q$: If~$x$ is a point contained in some critical fiber~$f^{-1}c$, then it is mapped to the~$0$--cell labeled by~$[x]$ in~$\pi_0f^{-1}c$. Otherwise, consider the~$1$--cell~$e$ corresponding to~$g_x$, the reparametrized flow-line provided by Proposition~\ref{proposition:modifiedFlowLines}. We then send~$x$ to the point in~$e$ mapped to~$f(x)$ under~$\bar{f}$. Note that this map is constructed to be over~$\R$ in the sense that~$f=\bar{f}\circ q$.

It only remains to verify that~$Q$ has the universal quotient topology--for the topological space~$Q$ is clearly in bijection with~$\mathrm{R}_f$. We thus verify that a subset~$U$ of~$Q$ is open if~$q^{-1}U$ is open in~$X$. This is true if for any closed~$1$--cell~$e$, the intersection~$e\cap U$ is open in~$e$. From the construction of~$Q$ it follows that there is a section~$\sigma\colon [c,d]\rightarrow X$ satisfying that~$q\circ \sigma\colon [c,d] \rightarrow Q$ corestricts to a homeomorphism~$e\simeq [c,d]$. But~$e\cap U$ corresponds to~$\sigma ([c,d])\cap q^{-1} U$ under the given homeomorphism, and~$\sigma ([c,d])\cap q^{-1} U$ is open in~$\sigma([c,d])$ given that~$q^{-1}U$ is open in~$X$.
\end{proof}

Let us end this section with a

\begin{proof}[proof of Theorem~\ref{intro:combinatorialReebIsClassicalReeb}]
Lemma~\ref{lemma:ReebhasReebData} allows us to assume that~$f\colon X\rightarrow \R$ induces a Reeb function~$\bar{f}\colon \reeb_f \rightarrow \R$. Theorem~\ref{intro:mainresult} thus guarantees~$\reeb_f\simeq \classifying \S_{\bar{f}}$. So it suffices to prove that the simplicial sets~$\pi_0\nerve\S_f$ and~$\nerve \S_{\bar{f}}$ are weakly equivalent.

Functoriality of~$\S$ induces a (continuous) functor~$\S_f\rightarrow \S_{\bar{f}}$ from the quotient map~$q\colon X\rightarrow \reeb_f$. It maps a section~$\sigma$ of~$f$ to the section~$q\circ \sigma$ of~$\bar{f}$. Sections in the same path components of~$(\nerve\S_f)_1$ are obviously mapped to the same section of~$\bar{f}$, so we have an induced simplicial map~$F\colon \pi_0\nerve \S_f\rightarrow \nerve \S_{\bar{f}}$. A morphism in~$\tau_1\pi_0\nerve\S_f$ is a word~$[\sigma_1]\cdots[\sigma_n]$, represented by sections~$\sigma_i$ which are composable up to paths contained in fibers of~$f$. The unit~$\eta\colon \pi_0\nerve\S_f\rightarrow\nerve\tau_1 \pi_0 \nerve\S_f$ provides a factorization
\[
\pi_0\nerve\S_f\xrightarrow{\eta}\nerve\tau_1 \pi_0 \nerve\S_f \xrightarrow{\nerve G}\nerve\S_{\bar{f}}
\]
of~$F$. We have already seen that~$\eta$ realizes to a weak homotopy equivalence in Lemma~\ref{lemma:addComposition}, so we need only verify that~$\nerve G$ is a weak equivalence. On the level of objects,~$G\colon\tau_1 \pi_0 \nerve\S_f \rightarrow \S_{\bar{f}} $ maps a morphism/word~$[\sigma_1]\cdots[\sigma_n]$ to the section~$(q\circ \sigma_n)\circ\dots\circ (q\circ \sigma_1)$ of~$\bar{f}$. We shall construct an inverse functor~$G^{-1}$ from which we conclude that~$\classifying G$ is in fact a homeomorphism.

Consider a section~$\rho\colon [c,d] \rightarrow \reeb_f$ of~$\bar{f}$ which passes no critical points, except possibly at the end points. Since~$\bar{f}\colon \reeb_f\rightarrow \R$ is assumed to be piecewise linear on a~$1$-dimensional CW complex, the image of~$\rho$ must be contained in some edge~$e$ of~$\reeb_f$. Take any point~$x$ in~$X$ that maps to the interior of~$e$ and define~$G^{-1}\rho$ to be~$[g_x|_{[c,d]}]$, the reparametrized flow-line provided by Proposition~\ref{proposition:modifiedFlowLines}. This class in~$\pi_0\S_f[c,d]$ is independent of the choice of~$x$. Indeed, assume that another point~$y$ is mapped to the interior of~$e$. Any choice of path~$p\colon I \rightarrow f^{-1}(c,d)$, between~$x$ and~$y$, defines a path from~$g_x|_{[c,d]}$ to~$g_y|_{[c,d]}$ in~$\S_f[c,d]$ via the composition
\[
[c,d]\times I\xrightarrow{\id_{[c,d]}\times p} [c,d]\times f^{-1}(c,d)\xrightarrow{g} X.
\]
A general section~$\rho$ of~$\bar{f}$ can only pass finitely many critical values, because~$f$ is Reeb. Whence we factorize it accordingly~$\rho=\rho_n\circ\cdots\circ \rho_1$ and define~$G^{-1}\rho$ to be the word~$G^{-1}\rho_1\cdots G^{-1}\rho_n$.

Applying~$G^{-1}\circ G $ to a word~$[\sigma_1]\cdots[\sigma_n]$ returns~$[g_{x_1}]\cdots [g_{x_n}]$, where~$x_i$ is chosen according to the above description of~$G^{-1}$. We can verify the equality~$[\sigma_i]=[g_{x_i}]$ by considering all reparametrized flow-lines through points in the image of~$[\sigma_i]$. Hence~$G ^{-1}\circ G=\id_{\tau_1\pi_0\nerve\S_f}$. The remaining equality~$G\circ G^{-1}=\id_{\S_{\bar{f}}}$ follows since the induced map~$\bar{f}\colon \reeb_f \rightarrow \R$ is piecewise linear on a~$1$--dimensional CW complex; an edge in~$\reeb_f$ uniquely determines a section that traverses it.
\end{proof}

\subsection{Combinatorial Reeb spaces are graphs}
\label{subsec:combisgraph}
Before we prove the result, I must first elaborate on the meaning of a `graph'.  A simplicial set is a \emph{graph} if it is aspherical--all higher homotopy groups are trivial--and the fundamental group is free for any choice of basepoint. The category~$(\text{graphs})$ of graphs is then the evident full subcategory of~$\sSet$. Our definition of combinatorial Reeb spaces gives a functor
\[
\over \rightarrow \sSet
\]
by mapping~$f\colon X \rightarrow \R$ to~$\pi_0\nerve\S_f$, and we shall see that it does in fact define a functor
\[
\over\rightarrow (\text{graphs}).
\]
In light of Example~\ref{example:ReebNotGraph}, I would argue that this is one advantage over topological Reeb spaces.

Classifying spaces of groupoids are aspherical, a fact which is easily verified by using simplicial homotopy groups. There is a groupoidification functor~$\Cat\rightarrow \Gpoid $ which assigns to a category~$\C$ the groupoid~$\C[\C^{-1}]$ in which all morphisms are formally inverted. It may abstractly be described as the left adjoint of the forgetful functor
\[
\Gpoid\rightarrow \Cat.
\]
For a given category~$\C$, there is an evident functor~$j\colon \C\rightarrow \C[\C^{-1}]$ which is not a weak homotopy equivalence in general: categories can represent all homotopy types, see e.g.~\cite{mcduff1979classifying}, whereas groupoids cannot. But we shall verify that~$j$ does in fact realize to a weak homotopy equivalence for combinatorial Reeb spaces.

Recall that a morphism in~$\tau_1\pi_0\nerve\S_f$ is a word~$[\sigma_1]\cdots[\sigma_n]$, represented by sections that are composable up to paths in fibers. A morphism in the groupoid~$\tau_1\pi_0\nerve\S_f[\tau_1\pi_0\nerve\S_f^{-1}]$ is thus a word~$[\sigma_1]^{i_1}[\sigma_2]^{i_2}\cdots[\sigma_n]^{i_n}$ in which each~$i_j=\pm 1$. The new relations imposed is generated by~$[\sigma][\rho]^{-1}= 1_{[s\sigma]}$ and~$[\rho]^{-1}[\sigma]=1_{[t\sigma]}$ whenever~$[\sigma]=[\rho]$; they are in the same component of the section space~$\S_f[f(s\sigma),f(t\sigma)]$. Geometrically you might want to interpret this as moving up and down, or down and up, along~$[\sigma]$ cancels to the appropriate identity. A word~$[\sigma][\rho]^{-1}$ is said to be \emph{reducible} if there  are factorizations~$[\sigma]=[\sigma_2]\circ [\sigma_1]$ and~$[\rho]=[\rho_1]\circ[\rho_2]$ such that~$[\sigma_2]=[\rho_2]$. In particular,~$[\sigma][\rho]^{-1}=[\sigma_1][\rho_1]^{-1}$. Dually, we declare what it means for~$[\sigma]^{-1}[\rho]$ to be reducible. Intuitively, a part of~$[\sigma]$ may overlap with~$[\rho]$:
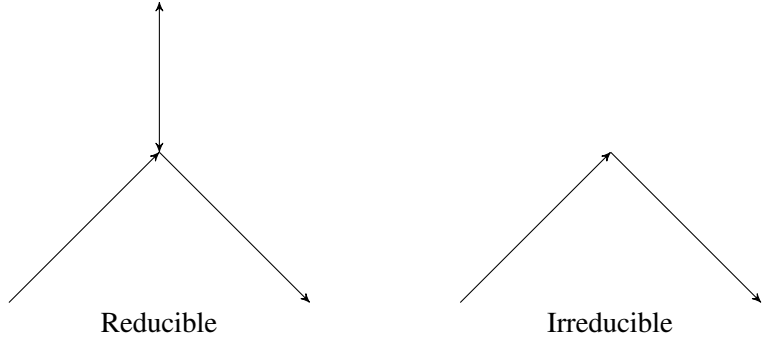
\begin{center}
\begin{tikzpicture}

\draw [->] (0,0) -- (2,2);

\draw [<-] (4,0) -- (2,2);

\draw [<->] (2,2) -- (2,4);

\draw [->] (6,0) -- (8,2);
\draw [<-] (10,0) -- (8,2);

\node [below] at (2,0) {Reducible};

\node [below] at (8,0) {Irreducible };

\end{tikzpicture}
\end{center}	
A morphism/word~$[\sigma_1]^{i_1}[\sigma_2]^{i_2}\cdots[\sigma_n]^{i_n}$ in~$\tau_1\pi_0\nerve\S_f[\tau_1\pi_0\nerve\S_f^{-1}]$ is then declared~\emph{irreducible} if the word has length equal to~$n$ and there is no reducible subword. Subject to this added requirement, we extend Lemma~\ref{lemma:dipath} to the groupoidification:

\begin{lemma}
\label{lemma:path}
Any morphism in~$\tau_1\pi_0\nerve\S_f[\tau_1\pi_0\nerve\S_f^{-1}]$ is uniquely presentable as an irreducible word.
\end{lemma}
\begin{proof}
Assume~$w$ to be presented~$[\sigma_1]^{i_1}\cdots[\sigma_n]^{i_n}$ and~$[\rho_1]^{j_1}\cdots[\rho_n]^{j_n}$, both irreducible. Let~$\bar{a}=(a_0,\dots,a_n)$ and~$\bar{b}=(b_0,\dots,b_n)$ be the sequences obtained by successively considering the domains of sections that appear as representatives in the two words.

These sequences must be equal and so the letters must be equal. Indeed, all relations connecting them alters the associated sequences of real numbers. The statement is clear if both are of length~$1$ since~$a_0=b_0$ and~$a_n=b_n$. If the length is~$\geq 2$ we conversely assume that~$\bar{a}\neq \bar{b}$. Consider~$q\geq 1$ the smallest index such that~$a_q\neq b_q$. We assume~$a_q<b_q$ without loss of generality.

Case 1:~$a_q<a_{q-1}$ and~$b_q>b_{q-1}$. Apply~$[\sigma_1]^{-i_1}\cdots[\sigma_{q-1}]^{-i_q}$ and~$w'=[\rho]_{n}^{-j_{n}}\cdots [\rho_{q+1}]^{-j_{q+1}}$ to~$w$. The result is an equality~$[\sigma_q]^{i_q}\cdots[\sigma_n]^{i_n}w'=[\rho_q]^{j_q}$. For this particular case, we deduce~$i_q=-1$ and~$j_q=1$. But then the equality can only hold if something cancels~$[\sigma_q]^{-1}$, contradicting the irreducibility of~$[\sigma_1]^{i_1}\cdots[\sigma_n]^{i_n}$.

Case 2:~$a_q>a_{q-1}$ and~$b_q > b_{q-1}$, or~$a_q < a_{q-1}$ and~$b_q< b_{q-1}$. These are proved in a similar fashion as the previous case.
\end{proof}

Before presenting the next Lemma I remind that a natural transformation~$ F\Rightarrow G$ between two functors~$\C\rightarrow \D$ is equivalent to a functor~$\C\times [1]\rightarrow \D$ whose restriction to~$0$ and~$1$ in~$[1]$ is~$F$ and~$G$, respectively. A natural transformation~$F\Rightarrow G$ thus defines a homotopy~$\classifying F\sim \classifying G$. See e.g. Segal's paper~\cite{segal1968classifying}. 
\begin{lemma}
\label{lemma:acyclic}
For any combinatorial Reeb space the associated map~$j\colon \tau_1\pi_0\nerve\S_f\rightarrow \tau_1\pi_0\nerve\S_f[\tau_1\pi_0\nerve\S_f^{-1}]$ realizes to a weak homotopy equivalence.
\end{lemma}
\begin{proof}
Consider an arbitrary object~$[x]$ in~$\tau_1\pi_0\nerve\S_f[\tau_1\pi_0\nerve\S_f^{-1}]$. Quilen's theorem~A reduces the problem to proving that the comma category~$j\downarrow [x]$ is contractible. An object in the comma category is a morphism/word in~$\tau_1\pi_0\nerve\S_f[\tau_1\pi_0\nerve\S_f^{-1}]$ terminating at~$[x]$. All words are presented uniquely according to Lemmas~\ref{lemma:dipath} and~\ref{lemma:path}. We shall define a homotopy from the identity on~$\classifying (j\downarrow [x])$ to the trivial map~$w\mapsto 1_{[x]}$. There are two essential intermediate functors.

The first functor~$\pr_+\colon j\downarrow [x] \rightarrow j\downarrow [x]$ reduces the length of words that start with a letter of the form~$[\sigma_1]$. It maps a non-trivial word~$w=[\sigma_1]^{i_1}\cdots[\sigma_n]^{i_n}$ to
\[
\pr_+ ([\sigma_1]^{i_1}\cdots[\sigma_n]^{i_n})=
\twopartdef{[\sigma_2]^{i_2}\cdots[\sigma_n]^{i_n}}{i_1=1}
{{[\sigma_1]^{i_1}\cdots[\sigma_n]^{i_n}}}{i_1=-1}
\]
A morphism~$w''\colon w\rightarrow w'$ is a factorization~$w=w'\circ j(w'')$ and hence there is a unique choice for~$\pr_+ w''$ yielding a factorization~$\pr_+w=\pr_+w' \circ (\pr_+w'')$. This data comes with a rather evident natural transformation~$\eta_+$ from~$ \id$ to~$ \pr_+$ since~$[\sigma_1]$ defines a morphism from a word~$[\sigma_1][\sigma_2]^{i_2}\cdots[\sigma_n]^{i_n}$ to~$[\sigma_2]^{i_2}\cdots[\sigma_n]^{i_n}$. In other words, we have defined a functor~$H_+\colon (j\downarrow [x]) \times [1]\rightarrow (j\downarrow [x]) $ whose restriction to~$(j\downarrow [x]) \times 0$ is~$\id$, whereas the restriction to~$(j\downarrow [x]) \times 1$ is~$\pr_+$.

The second functor~$\pr_-\colon j\downarrow [x] \rightarrow j\downarrow [x]$ is complementary to~$\pr_+$. It maps a non-trivial word~$w=[\sigma_1]^{i_1}\cdots[\sigma_n]^{i_n}$ to
\[
\pr_- ([\sigma_1]^{i_1}\cdots[\sigma_n]^{i_n})=
\twopartdef{[\sigma_2]^{i_2}\cdots[\sigma_n]^{i_n}}{i_1=-1}
{{[\sigma_1]^{i_1}\cdots[\sigma_n]^{i_n}}}{i_1=1}
\]
Analogous to~$\pr_+$ this data comes with a homotopy~$H_-\colon (j\downarrow [x]) \times [1]\rightarrow (j\downarrow [x]) $. But, in contrast to~$H_+$, this homotopy starts at~$\pr_-$ and terminates at~$\id$. This is because of how~$[\sigma_1]$ defines a morphism from~$[\sigma_2]^{i_2}\cdots[\sigma_n]^{i_n}$ to~$[\sigma_1]^{-1}[\sigma_2]^{i_2}\cdots[\sigma_n]^{i_n}$.

For every object~$w$, we need only alternate~$H_+$ and~$H_-$ a finite number of times to obtain the trivial word~$1_{[x]}$. We thus conclude that the identity on~$(j\downarrow [x])_n$, generated by words of length~$\leq n$, is homotopic to the trivial map for all~$n$. It follows that the identity on~$j\downarrow [x]$ must be homotopic to the trivial map.
\end{proof}

We wrap up the discussion on Combinatorial Reeb spaces with a
\begin{proof}[proof of Theorem~\ref{intro:combinatorialReebIsGraph}]
We have seen in Lemma~\ref{lemma:acyclic} that a Reeb space~$\pi_0\nerve\S_f$ has the homotopy type of its groupoidification. In particular, it must be aspherical. It remains only to verify that the fundamental group is free, regardless of basepoint. So fix a basepoint~$[x]$ and consider~$\pi_1 (\pi_0\nerve\S_f)$ which is isomorphic to the automorphism group at~$[x]$, considered as an object in the groupoidification~$\tau_1\pi_0\nerve\S_f[\tau_1\pi_0\nerve\S_f^{-1}]$. This group admits the irreducible words (Lemma~\ref{lemma:path}) as a free generating set. For a non-trivial irreducible word cannot possibly be reduced further to the unit~$1_{[x]}$.
\end{proof}

\section*{Acknowledgment}
I would like to thank Markus Szymik who helped shape the ideas and contents of this paper through many helpful discussions.

\addcontentsline{toc}{section}{References}
\bibliographystyle{amsalpha}
\bibliography{CombinatorialReebSpaces}

\end{document}